\documentclass[11pt]{article}
\usepackage{color}
\usepackage{amsfonts}
\usepackage{graphicx}
\usepackage{amsmath,enumerate}
\usepackage{cases}
\usepackage{amsthm}
\usepackage{amssymb}
\usepackage{amsfonts}
\usepackage{latexsym,amssymb,amsfonts}
\usepackage{mathrsfs}
\usepackage{graphicx}
\usepackage{psfrag}
\usepackage{amsmath, amsbsy}
\usepackage{amsopn, amstext}
\usepackage{amsmath,multicol,amsopn}
\usepackage{latexsym,amssymb}
\usepackage{hyperref}

\usepackage{lineno}
\def\dbF{{\mathbb{F}}}

\def\dbN{{\mathbb{N}}}

\def\dbP{{\mathbb{P}}}

\def\dbR{{\mathbb{R}}}
\def\dbS{{\mathbb{S}}}

\def\e{\varepsilon}

\def\th{\theta}
\def\Th{\Theta}

\def\3n{\negthinspace \negthinspace \negthinspace }
\def\2n{\negthinspace \negthinspace }
\def\1n{\negthinspace }
\def\ns{\noalign{\smallskip} }

\def\ds{\displaystyle}
\def\G{\Gamma}

\def\Th{\Theta}

\def\Om{\Omega}

\def\cA{{\cal A}}

\def\cC{{\cal C}}

\def\cF{{\cal F}}

\def\cJ{{\cal J}}

\def\cL{{\cal L}}

\def\cP{{\cal P}}

\def\cU{{\cal U}}

\def\mE{{\mathbb{E}}}

\def\no{\noindent}

\def\ss{\smallskip}
\def\ms{\medskip}
\def\bs{\bigskip}
\def\q{\quad}
\def\qq{\qquad}

\def\liminf{\mathop{\underline{\rm lim}}}

\def\wt{\widetilde}
\def\cd{\cdot}
\def\cds{\cdots}

\def\ae{\hbox{\rm a.e.{ }}}
\def\as{\hbox{\rm a.s.{ }}}

\def\({\Big (}
\def\){\Big )}
\def\[{\Big[}
\def\]{\Big]}

\def\={\buildrel \triangle \over =}

\def\ee{\end{equation}}
\def\bea{\begin{eqnarray}}
\def\eea{\end{eqnarray}}
\def\bt{\begin{theorem}}
\def\et{\end{theorem}}
\def\bc{\begin{corollary}}
\def\ec{\end{corollary}}
\def\bl{\begin{lemma}}
\def\el{\end{lemma}}
\def\bp{\begin{proposition}}
\def\ep{\end{proposition}}
\def\br{\begin{remark}}
\def\er{\end{remark}}
\def\ba{\begin{array}}
\def\ea{\end{array}}
\def\bde{\begin{definition}}
\def\ede{\end{definition}}

\newtheorem{lemma}{Lemma}[section]
\newtheorem{remark}{Remark}[section]
\newtheorem{example}{Example}[section]
\newtheorem{theorem}{Theorem}[section]
\newtheorem{corollary}{Corollary}[section]

\newtheorem{definition}{Definition}[section]
\newtheorem{proposition}{Proposition}[section]

\newtheorem{assumption}{Assumption}[section]

\allowdisplaybreaks[4]
\makeatletter

\@addtoreset{equation}{section}
\makeatother

\hypersetup{
hidelinks
}

\title{\bf Time-inconsistent Linear Quadratic Optimal Control Problem for Forward-Backward Stochastic Differential Equations\footnote{This work is supported by the NSF of China under grants 12025105, 11971333 and 11931011, and by the Science Development Project of Sichuan University under grant 2020SCUNL201.} }

\author{Qi L\"u \footnote{School of Mathematics, Sichuan University, Chengdu, 610064, China  (lu@scu.edu.cn)}\quad and \quad Bowen Ma \footnote{College of Mathematics and Physics, Chengdu University of Technology, Chengdu, 610059, China  (albertmabowen@gmail.com)} }

\date{}

\begin{document}
\maketitle
\begin{abstract}
We study the time-inconsistent linear quadratic optimal control problem  for forward-backward stochastic differential equations with potentially indefinite cost weighting matrices for both the state and the control variables. Our research makes two contributions.  Firstly, we introduce a novel type of Riccati equation system with parameters and constraint conditions, known as the generalized equilibrium Riccati equation. This equation system offers a comprehensive solution for the closed-loop equilibrium strategy of the problem at hand. Secondly, we establish the well-posedness of the generalized equilibrium Riccati equation for the one-dimensional case, provided certain conditions are met.
\end{abstract}

\no{\bf 2020 Mathematics Subject
	Classification}. Primary 93E20, 49N10.

\bs

\no{\bf Key Words}. stochastic linear quadratic
control problem, forward-backward stochastic differential equation, time-inconsistency.

\section{Introduction}\label{sec-intro}

Let $(\Omega,\cF,\dbF,\dbP)$ be a complete  filtered probability space, on which a standard  one-dimensional   Brownian motion $W(\cd)$ is defined, and $\dbF=\{\cF_t \}_{t\geq 0}$ is  the natural filtration generated by $W(\cd)$.

For a matrix $M$, we use $M^\dagger$
to represent the Moore-Penrose inverse of $M$ and $M^\top$
to represent the transpose of $M$. In addition, for a positive integer $n$, $\dbS^n$
denotes the space consisting of all $n\times n$ symmetric matrices, and $\dbS^n_+$ 
denotes the space consisting of all positive semidefinite matrices in $\dbS^n$. Furthermore, $|M|_\infty$ 
represents the spectral norm of a matrix $M\in \dbR^{n\times m}$, which is equal to the square root of the largest eigenvalue of $M^\top M$. 

We define some Banach spaces for $k\in \dbN$, $t\in [0,T]$, and $p\in[1,\infty)$, $q\in [1,\infty)$.

Firstly, $L_{\cF_t}^p(\Omega;\dbR^k)$ is the Banach space consisting of all $\cF_t$-measurable random variables $\xi: \Omega\to \dbR^k$
such that $\mE |\xi|_{\dbR^k}^p<\infty$, with the canonical norm.

Next, $L_{\dbF}^p(\Omega;C([t,T];\dbR^k))$ is the Banach space of all $\dbR^k$-valued $\dbF$-adapted continuous stochastic processes $\phi(\cdot)$. The norm for this space is defined as $|\phi(\cdot)|_{L_{\dbF}^p(\Omega;C([t,T];\dbR^k))} = \left(\mE \max\limits_{s\in[t,T]}|\phi(s)|_{\dbR^k}^p \right)^{1/p}$.

Lastly, $L_{\dbF}^p(\Omega;L^q(t,T;\dbR^k))$ is a Banach space consisting of $\dbR^k$-valued $\dbF$-adapted stochastic processes $\phi(\cdot)$ defined on the product space $\Omega\times [t,T]$ such that $\mE \left( \int_{t}^{T}|\phi(s)|^q_{\dbR^k}ds \right)^{p/q} < \infty$. The norm of this Banach space is given by the canonical norm. We can also write $L_{\dbF}^p(\Omega;L^p(t,T;\dbR^k))$
simply as $L^p_{\dbF}(t,T;\dbR^k)$.

Let $T>0$ and $(t,x)\in [0,T)\times L_{\cF_t}^2(\Omega;\dbR^n)$. Consider the following controlled \textit{linear forward-backward stochastic differential equation} (FBSDE, for short) on the time horizen $[t,T]$:
\begin{equation}\label{s4.1-stat}
\begin{cases}\ds
dX(s)=(A(s)X(s)+B(s)u(s))ds +(C(s)X(s)+D(s)u(s))dW(s), & s\in [t,T]\\ \ns\ds
dY(s)=-\big(\widehat{A}(s)X(s)+\widehat{B}(s)u(s)+\widehat{C}(s)Y(s)  +\widehat{D}(s)Z(s)\big)ds+Z(s)dW(s), & s\in [t,T]\\  \ns\ds
X(t)=x,\q Y(T)=HX(T),
\end{cases}
\end{equation}
where $A(\cd),C(\cd)\in L^\infty(0,T; \dbR^{n\times n})$, $B(\cd),D(\cd)\in L^\infty(0,T; \dbR^{n\times k}),$ $\widehat{A}(\cd)\in L^\infty(0,T;\dbR^{m\times n})$, $\widehat{B}(\cd)\in  L^\infty(0,T; \dbR^{m\times k})$, and $\widehat{C}(\cd),\widehat{D}(\cd)\in L^\infty(0,T; \dbR^{m\times m})$, $H\in \dbR^{m\times n}$, and the control $u(\cd) \in \cU[t,T]\= L_{\dbF}^2(t,T;\dbR^{k})$.

The equation \eqref{s4.1-stat} consists of a (forward) stochastic differential equation (SDE, for short) and a backward stochastic differential equation (BSDE, for short). Since $Y(\cdot)$ and $Z(\cdot)$ do not appear in the SDE in \eqref{s4.1-stat}, one can first solve the SDE to obtain $X(\cdot)$, and then solve the BSDE to obtain $(Y(\cdot), Z(\cdot))$. Therefore, based on the classical well-posedness of SDEs and BSDEs, it is known that for any $(t, x)\in [0,T)\times L_{\mathcal{F}_t}^2(\Omega;\mathbb{R}^n)$  and $u(\cdot)\in \mathcal{U}[t,T]$, the equation \eqref{s4.1-stat} has a unique solution $(X(\cdot), Y(\cdot), Z(\cdot))\in L_{\mathbb{F}}^2(\Omega; C([t,T];\mathbb{R}^n))\times L_{\mathbb{F}}^2(\Omega; C([t,T];\mathbb{R}^m))\times L_{\mathbb{F}}^2(t,T;\mathbb{R}^m)$  (e.g., \cite[Sections 3.1 and 4.1]{Lu-2021}).

We introduce the following  cost functional 
\begin{equation}\label{s4.1-cost}
\begin{aligned}
\cJ(t,x;u(\cd))=&\frac{1}{2}\mE_t\Big[ \int_{t}^{T} \big(\langle Q(s,t)X(s),X(s) \rangle +\langle R(s,t)u(s),u(s) \rangle+ \langle M(s,t)Y(s),Y(s) \rangle  \\ & \qq +  \langle N(s,t)Z(s),Z(s) \rangle \big)ds+  \langle G_1(t)X(T),X(T) \rangle + \langle G_2(t)Y(t),Y(t) \rangle \Big],
\end{aligned}
\end{equation}
where 	$\mE_t:=\mE(\cd \,|\, \cF_t)$ denotes the conditional expectation with respect to $\cF_t$,  $G_1(\cd)\in C([0,T];\dbR^n), \\G_2(\cd)\in C([0,T];\dbR^m)$,   $Q(\cd,\cd)\in C([0,T]^2; \dbS^{n}),$ $M(\cd,\cd),N(\cd,\cd)\in C([0,T]^2; \dbS^{m})$ and  $R(\cd,\cd)\in C([0,T]^2; \dbS^{k})$.

For  any  $u(\cd)\in \cU[t,T]$, the cost functional $\cJ(t,x;u(\cd))$ is well defined.  Naturally, we can pose a standard optimal control problem as follows:

\ss

\textbf{Problem (TI-FBSLQ).} For any $(t,x)\in [0,T)\times L^2_{\cF_t}(\Om;\dbR^n)$, find a control $\bar{u}(\cd)\in \cU[t,T]$ such that
\begin{equation}\label{s4.1-essinf}
\cJ(t,x;\bar{u}(\cdot))=\inf_{u(\cdot)\in
\mathcal{U}[t,T]}\cJ(t,x;u(\cdot)).
\end{equation}
\begin{remark}
Recently, Wang et al. \cite{H.Wang-2022} made an interesting discovery regarding optimal control problems for FBSDEs, that is, they found that these problems exhibit a characteristic of being \textit{time-inconsistent}. In light of this, we refer to the aforementioned problem as the TI-FBSLQ problem, with the abbreviation "TI" signifying the \textit{time-inconsistency}. 
\end{remark}

For Problem (TI-FBSLQ), a control $\bar{u}(\cdot)\in\mathcal{U}[t,T]$ that satisfies \eqref{s4.1-essinf} is referred to as a ``pre-committed optimal control". Although the pre-committed optimal control $\bar{u}(\cdot)$ is optimal for the cost functional $\mathcal{J}(t,x;\cdot)$ for any fixed initial pair $(t,x)$, it may not be practical in practice as it may not remain optimal at later times. There are two main reasons in order:
\begin{itemize}
\item[\textbf{(1)}] In the cost functional \eqref{s4.1-cost}, the coefficients $Q(\cd,t),R(\cd,t),M(\cd,t),N(\cd,t),G_1(t),G_2(t)$ depend on the initial time $t$. This subjective time-preference leads to time-inconsistent behavior (see \cite{Dou-2020,Lu-2023,Yong-2014,Yong-2017} and the rich reference therein).  For example, if we assume:
$$
\begin{cases}\ds
H=0,\,\widehat{A}(\cd)=\widehat{C}(\cd)=0,\widehat{B}(\cd)=\widehat{D}(\cd)=0,\\
\ns\ds	M(\cd,\cd)=0,\q N(\cd,\cd)=0,\q G_2(\cd)=0,
\end{cases}
$$
then Problem (TI-FBSLQ) will reduce to the model  in \cite{Lu-2023,Yong-2017}, where an illustrative example is provided.

\item [\textbf{(2)}]  Even if there is no subjective time-preference, i.e.,  the coefficients $Q(\cd,t),R(\cd,t),M(\cd,t),\\N(\cd,t),G_1(t),G_2(t)$ do not depend on the initial time $t$,  Problem (TI-FBSLQ) remains time-inconsistent due to the forward-backward structure of the control system \eqref{s4.1-stat} (e.g., \cite{H.Wang-2022}). For instance,  consider the following simple example: 
$$
\begin{aligned}
m=n,\,\; H=I,\,\; \widehat{A}(\cd)=\widehat{C}(\cd)=0,\,\; \widehat{B}(\cd)=\widehat{D}(\cd)=0,\,\;  G_2(\cd)\not=0.
\end{aligned}
$$
In the cost functional \eqref{s4.1-cost},   the term
$
\langle G_2(t) Y(t),Y(t)\rangle=\langle G_2(t)\mE_t(X(T)),\mE_t(X(T))\rangle
$ appears,   which naturally results in time-inconsistency (see \cite{Hu-2012, Hu-2017, Yong-2017}). 
\end{itemize}

To address the issue of time inconsistency, we propose the following definition.

\begin{definition}\label{s4.1-de1}
A matrix-valued function ${\Theta}(\cdot)\in L^2 (0,T;\mathbb{R}^{k\times n})$ is called the closed-loop equilibrium strategy of Problem (TI-FBSLQ) if for any sequence
$\{\e_j\}_{j=1}^\infty\subset (0,+\infty)$ converging to $0$, $(t,x)\in [0,T)\times L_{\mathcal{F}_t}^2(\Omega;\mathbb{R}^n)$ and
$v\in L_{\mathcal{F}_t}^2(\Omega;\mathbb{R}^k)$, 
\ss
\begin{equation}\label{}
\varliminf_{j \to \infty} \frac{\cJ\left(t, \overline{X}(t) ;
u^{\varepsilon_j}(\cdot)\right)-\cJ\left(t, \overline{X}(t) ;
\bar{u}(\cdot)\right)}{\varepsilon_j} \geq 0,\qquad
\mathbb{P}\mbox{\rm-a.s.}. 
\end{equation}
Here 
\begin{equation}
\begin{cases}
\ds	d\overline{X}(s)=(A(s)+B(s)\Theta(s))\overline{X}(s)ds +(C(s)+D(s)\Theta(s))\overline{X}(s)dW(s), & s\in [t,T], \\
\ns\ds	d\overline{Y}(s)=-\big[ \big(\widehat{A}(s) + \widehat{B}(s)\Theta(s)\big)\overline{X}(s)+\widehat{C}(s)\overline{Y}(s) +\widehat{D}(s)\overline{Z}(s)\big]ds+\overline{Z}(s)dW(s), & s\in [t,T],\\
\ns\ds	\overline{X}(t)=x,\q \overline{Y}(T)=H\overline{X}(T),
\end{cases}
\end{equation}
\begin{equation}\label{s4.1-de1-eq1}
\bar{u}(s)=\Theta(s)\overline{X}(s),\q
u^{\e_j}(s)=\chi_{[t,t+\e_j]}(s)v+{\Theta}(s){X}^{\e_j}(s),
\end{equation}
and $ X^{\e_j}(\cd)$ is the solution of equation \eqref{s4.1-stat}
corresponding to the control $u^{\e_j}(\cdot)$. 
\end{definition}
\begin{remark}
The variation of control $\bar{u}(\cdot)$ in Definition \ref{s4.1-de1} is given by $\chi_{[t,t+\e_j]}v+\Theta(s)X^{\e_j}(s)$, which is different from the general form $\chi_{[t,t+\e_j)}u(\cdot)$ used in \cite{Dou-2020, Lu-2023}. Although Definition \ref{s4.1-de1} is weaker, it is still useful, as shown in \cite{Hu-2012, Hu-2017}. We adopt this definition to facilitate the analysis of our problem using the variational approach. 
\end{remark}

In what follows, for simplicity of notations, for ${\Theta}(\cdot)\in L^2 (0,T;\mathbb{R}^{k\times n})$, we denote
\begin{equation*}
\begin{cases}
\ds A_\Theta(s):=A(s)+B(s)\Theta(s) ,\\ \ns\ds  C_\Theta(s):= C(s)+D(s)\Th(s), \\ \ns\ds \widehat{A}_\Theta(s):=\widehat{A}(s)+\widehat{B}(s)\Th(s), 
\end{cases}
\end{equation*}

To study the closed-loop equilibrium strategy of Problem (TI-FBSLQ), we need to introduce  the following assumption:
\begin{assumption}\label{s4.1-H1}
There exists a constant $\cC>0$, such that  for any $0\leq t\leq \tau \leq s\leq T$, it holds that
\begin{equation}
\begin{aligned}
&|Q(s,t)-Q(s,\tau)|_\infty + |R(s,t)-R(s,\tau)|_\infty+ |M(s,t)-M(s,\tau)|_\infty \\&\, +|N(s,t)-N(s,\tau)|_\infty+|G_1(t)-G_1(\tau)|_\infty+|G_2(t)-G_2(\tau)|_\infty \leq \mathcal{C} |t-\tau|.  
\end{aligned}
\end{equation}
\end{assumption}
Here and in what follows,  to simplify the expression, we make two conventions when there is no ambiguity: First, we denote by $\cC$ a generic constant which may vary from  line to line; Second,  when a function in the equation  involves only a  single time variable, we omit it, whereas with two time variables, we explicitly  specify them.

\ss 

For $\theta_0 (\cd) \in L^2(0,T;\dbR^{k\times n})$, consider the following equation:
\begin{equation}\label{s4.1-cor1-eq1}
\begin{cases}\ds
\frac{dP_1(s;t)}{ds}+ P_{1}(s;t) A_\Theta  +A_\Theta^\top  P_1(s;t) \\\ns\ds\q   +C_\Theta^\top  P_{1}(s;t)C_\Theta +Q(s,t) +\Th^\top R(s,t)\Th=0, & 0\leq t\leq s\leq T,\\
\ns\ds \frac{dP_2}{ds}+ P_2A_\Theta+  \widehat{A}_\Theta  +\widehat{C}P_2 +\widehat{D} P_2C_\Theta =0, & 0\leq t\leq s\leq T,\\
\ns\ds \frac{ dP_3(s;t)}{ds} + P_3(s;t)A_\Theta 	+A_\Theta ^\top P_3(s;t)   +C_\Theta ^\top P_3(s;t)C_\Theta \\ \ns\ds \q    +P_2^\top M(s,t)P_2  + C_\Theta ^\top P_2^\top N(s,t) 	P_2 C_\Theta  =0, &  0\leq t\leq s\leq T,\\
\ns\ds P_1(T;t)=G_1(t),	\q P_2(T)=H, \q 	P_3(T;t)=0, & 0\leq t\leq T.
\end{cases}
\end{equation}
where
\begin{equation}\label{s4.1-cor1-eq2}
\begin{aligned}
\Theta(t)&=-\big[R(t,t)+D(t)^\top
\big(P_1(t;t)  \! +  \! P_3(t;t)  \! +  \! P_2(t)^\top N(t,t)P_2(t) \big)D(t)\big]^{\dagger}   \\&\q  \times \big[B(t)^\top\big( P_1(t;t)  \! +  \! P_3(t;t)\big)  +D(t)^\top
\big(P_1(t;t) +P_3(t;t)   +P_2(t) ^\top N(t,t)P_2(t) \big)C(t)\\&\qq   + \big(\widehat{B}(t)^\top+ {B}(t)^\top P_{2}(t)^\top  + D (t)^\top P_{2}(t)^\top  \widehat{D}(t) ^\top \big ) G_2(t)P_{2}(t) \big] + \th_0(t)  \\& \q-\big[R(t,t)  \! +  \! D(t) ^\top
\big(P_1(t;t)+ P_3(t;t)    \!  + \! P_2(t)^\top N(t,t)P_2(t) \big)D(t)\big]^{\dagger} \\& \q\times \big[R(t,t) \!+ \! D(t) ^\top
\big(P_1(t;t)  \! +  \! P_3(t;t)  \!+ \! P_2(t)^\top  \! N(t,t)P_2(t) \big)D(t)\big]\th_0(t).
\end{aligned}
\end{equation}

Our first main result is as follows.
\begin{theorem}\label{s4.1-cor1}
Let Assumption \ref{s4.1-H1} hold. Then   Problem (TI-FBSLQ)  admits a closed-loop equilibrium strategy if and only if    there exists  $\theta_0 (\cd) \in L^2(0,T;\dbR^{k\times n})$ such that the   equation \eqref{s4.1-cor1-eq1} admits a solution $(P_1(\cd;\cd),P_2(\cd),P_3(\cd;\cd))$ satisfying
\begin{equation}\label{s4.1-cor1-eq3}
\begin{array}{ll} \ds
\big[R(\cd,\cd)+D(\cd) ^\top
\big(P_1(\cd;\cd)+ P_3(\cd;\cd)+P_2(\cd)^\top N(\cd,\cd)P_2(\cd) \big)D(\cd)\big]^{\dagger}\\ \ns\ds \times \big\{\big[B(\cd)^\top\big( P_1(\cd;\cd)+P_3(\cd;\cd)\big)+D(\cd)^\top
\big(P_1(\cd;\cd) +P_3(\cd;\cd)+P_2(\cd)^\top N(\cd,\cd)P_2(\cd) \big)C(\cd)\big]    \\\ns\ds\, +\big(\widehat{B}(\cd)^\top+ {B}(\cd)^\top P_{2}(\cd)^\top + D^\top (\cd) P_{2}(\cd)^\top  \widehat{D}(\cd)^\top\big ) G_2(\cd)P_{2}(\cd) \big\} \,\in \,  L^2(0,T;\dbR^{k\times n}),
\end{array}
\end{equation}
\begin{equation}\label{s4.1-cor1-eq4}
\begin{array}{ll} \ds
\mathcal{R}\Big( R(t,t)+D(t)^\top
	\big(P_1(t;t)+ P_3(t;t)+P_2(t)^\top N(t,t)P_2(t) \big)D(t) \Big)  \supseteq\\ \ns \ds  \mathcal{R}\Big( B(t)^\top\big( P_1(t;t)+P_3(t;t)\big)+D(t)^\top
	\big(P_1(t;t) +P_3(t;t)+P_2(t) ^\top N(t,t)P_2(t) \big)C(t)   \\ \q \,\,\, +\big(\widehat{B}(t)^\top+ {B}(t)^\top P_{2}(t)^\top + D (t)^\top P_{2}(t) ^\top \widehat{D}(t)^\top\big ) G_2(t)P_{2}(t) \Big), \q \ae  t\in[0,T],
\end{array}
\end{equation}
and 
\begin{equation}\label{s4.1-cor1-eq5}
	\begin{array}{ll}
		R(t,t)\! + \!  D(t)^\top P_1(t;t) D(t)   +   D(t)^\top P_3(t;t) D(t)   + D(t)^\top P_2(t)^\top N(t,t)	P_2(t)  D(t)     \geq 0,  \q \ae  t\in[0,T].
	\end{array}
\end{equation}
In this case,  $\Theta(\cd)$ given by \eqref{s4.1-cor1-eq2} is a closed-loop equilibrium strategy.
\end{theorem}
\begin{remark}
If
\begin{equation}\label{rm2.3-eq1}
H=0,\; \widehat{A}(\cd)=\widehat{C}(\cd)=0,\;\widehat{B}(\cd)=\widehat{D}(\cd)=0,\; M(\cd,\cd)=N(\cd,\cd)=G_2(\cd)=0,
\end{equation} 
then the Problem (TI-FBSLQ) can be reduced to the Problem (TISLQ) that was studied in the reference \cite{Lu-2023}. The equations \eqref{s4.1-cor1-eq1}--\eqref{s4.1-cor1-eq2} mentioned in the problem are consistent with the equations derived in \cite[Theorem 2.10]{Lu-2023}. It is worth noting that in \cite{Lu-2023}, the equations were obtained using the approach of multiperson differential games. However, in our case, we use a different approach, namely, the variational approach, to obtain the equations \eqref{s4.1-cor1-eq1}--\eqref{s4.1-cor1-eq2}.
\end{remark}
\begin{remark}
We will refer to equations \eqref{s4.1-cor1-eq1}--\eqref{s4.1-cor1-eq5} as the ``generalized equilibrium Riccati equation" for Problem (TI-FBSLQ). This term is used because when condition \eqref{rm2.3-eq1} is satisfied and the time-preference disappears (meaning that the coefficients $Q(\cd,t)$, $R(\cd,t)$, $G_1(t)$  do not depend on the initial time $t$), our problem can be naturally reduced to a classical indefinite stochastic LQ problem. In this case, equations \eqref{s4.1-cor1-eq1}--\eqref{s4.1-cor1-eq5} become equivalent to the classical generalized Riccati equation as discussed in references \cite{Chen-1998,Rami-2001,Sun-Yong-2020}.
\end{remark}

In Theorem \ref{s4.1-cor1}, the solvability of equations \eqref{s4.1-cor1-eq1}--\eqref{s4.1-cor1-eq3} is dependent on the choice of $\theta_0$. This may appear to be counterintuitive. However, the following example demonstrates that such dependence is indeed necessary in Theorem \ref{s4.1-cor1}.
\begin{example}
Let  \eqref{rm2.3-eq1} hold, 
$n=k=1$, and
\begin{equation*}
\begin{cases}
A(\cd)=D(\cd)=0,\, B(\cd)=C(\cd)=1,\,R(s,t)=s-t,\\\ns\ds
Q(s,t)\geq 0, \q G_1(t)=-\int_{t}^{T}e^{-(T-\tau)}Q(\tau,t)d\tau.
\end{cases}
\end{equation*}

When we choose $\theta_0(\cd)=0$, the equations \eqref{s4.1-cor1-eq1}--\eqref{s4.1-cor1-eq5} are solvable. In this case,  we can get $P_6(\cd)=0$, $P_8(\cd;\cd)=0$, and
\begin{equation*}
\begin{cases}
\ds \frac{dP_1(s;t)}{ds}	=-P_1(s;t)-Q(s,t), & t\leq s\leq T,\\
\ns\ds	P_1(T;t)=G_1(t), & 0\leq t\leq T.
\end{cases}
\end{equation*}
Therefore $P_1(t;t)=0$, and the  constraint  conditions \eqref{s4.1-cor1-eq3}--\eqref{s4.1-cor1-eq5} hold. By Theorem \ref{s4.1-cor1}, the closed-loop equilibrium strategies exist.

When we choose $\theta_0(\cd)=-\frac{1}{2}$, the equations \eqref{s4.1-cor1-eq1}--\eqref{s4.1-cor1-eq2}  are  solvable, but the  constraint  conditions \eqref{s4.1-cor1-eq4} fail. Indeed, in this case,  we can get $P_6(\cd)=0,P_8(\cd;\cd)=0,$ and
\begin{equation*}
\begin{cases}\ns\ds
 \frac{ dP_1(s;t)}{ds}=-Q(s,t)-\frac{s-t}{4} , & t\leq s\leq T,\\\ns\ds 
P_1(T;t)=G_1(t), & 0\leq t\leq T.
\end{cases}
\end{equation*}
Hence
\begin{equation*}
\begin{aligned}
P_1(t;t)&=G_1(t)-\int_{t}^{T}\big[-Q(\tau,t)-(\tau-t)/4\big]d\tau\\
&=-\int_{t}^{T}e^{-(T-\tau)}Q(\tau,t)d\tau+\int_{t}^{T}\big[Q(\tau,t)+(\tau-t)/4\big]d\tau\\
&=\int_{t}^{T}\big(1-e^{-(T-\tau)} \big)Q(\tau,t)d\tau +(T-t)^2/8>0.
\end{aligned}
\end{equation*}
We can see that
\begin{equation*}
\begin{aligned}
\mathcal{R}\left( R(t,t)+D(t)^\top P_1(t;t)D(t) \right)\equiv	\mathcal{R}\left( 0 \right) \not \supseteq \mathcal{R}\left( B(t)^\top P_1(t;t)+D(t) ^\top P_1(t;t)C(t) \right).
\end{aligned}
\end{equation*}
Consequently, \eqref{s4.1-cor1-eq4} does not hold.
\end{example}

This example demonstrates that, in general, for Problem (TI-FBSLQ), the solvability of the equations \eqref{s4.1-cor1-eq1}--\eqref{s4.1-cor1-eq5} is dependent on the choice of the parameter $\th_0(\cd)$, even if its closed-loop equilibrium strategy exists.

Theorem \ref{s4.1-cor1} establishes the connection between the existence of closed-loop equilibrium strategies and the solvability of the system \eqref{s4.1-cor1-eq1}--\eqref{s4.1-cor1-eq5}. It is natural to inquire about the conditions under which the system \eqref{s4.1-cor1-eq1}--\eqref{s4.1-cor1-eq5} is solvable. Unfortunately, as of now, we only have the answer to this question for the one-dimensional case, which is the second main result of this paper. To present it, we first make the following assumption:
\begin{assumption}\label{s4.1-H3}
Let $m=n=k=1$.  
There exists a constant $\delta>0$, such that
$$
\begin{cases}\ds
R(t,t)\geq \delta, \qq  t\in [0,T], \\
\ns\ds N(t,t)\geq \delta,\qq  t\in [0,T],  \\ \ns\ds D(t)^\top D(t)\geq \delta,\q \ae t\in [0,T]. 
\end{cases}
$$
Moreover,
$$\begin{cases}\ds	
Q(s,t)\geq 0, \\
\ns\ds  M(s,t)\geq 0,  \\
\ns\ds  G_1(t)\geq 0,  
\end{cases}\q 0\leq t\leq s\leq T.
$$
\end{assumption}	
The  second main result of this paper is as follows:
\begin{theorem}\label{th2}
Let Assumptions \ref{s4.1-H1}--\ref{s4.1-H3} hold.  Then for any  $\theta_0 (\cd) \in L^2(0,T;\dbR)$, the  system  \eqref{s4.1-cor1-eq1}--\eqref{s4.1-cor1-eq5}  admits a unique solution $(P_1(\cd;\cd),P_2(\cd),P_3(\cd;\cd))$. Consequently,  Problem (TI-FBSLQ) admits a unique equilibrium strategy given by 
\begin{equation}\label{th2-eq1}
\begin{aligned}
\Theta(t)&=\! -\big[R(t,t)+D(t)^\top
\big(P_1(t;t)  \! +  \! P_3(t;t)  \! +  \! P_2(t)^\top N(t,t)P_2(t) \big)D(t)\big]^{\dagger} \\&\q \times \big[ B(t)^\top\big( P_1(t;t)  \! +  \! P_3(t;t)\big)  \!+\! D(t)^\top \!
\big(P_1(t;t) \!+\! P_3(t;t)\! +\! P_2(t) ^\top\! N(t,t)P_2(t) \big)\\&\qq \times C(t)     +\! \big(\widehat{B}(t)^\top \!\!   + \!\!{B}(t)^\top\! P_{2}(t)^\top \!\! + \!\! D (t)^\top \! P_{2}(t)^\top  \widehat{D}(t) ^\top \big )   G_2(t)P_{2}(t) \big].
\end{aligned}
\end{equation}	
\end{theorem}

\begin{remark}
In fact, under  Assumptions \ref{s4.1-H1}--\ref{s4.1-H3},  we would have
\begin{align*}
&	\big[R(t,t)  \! +  \! D(t) ^\top
\big(P_1(t;t)+ P_3(t;t)    + \! P_2(t)^\top N(t,t)P_2(t) \big)D(t)\big]^{\dagger} \\ &\times \big[R(t,t) \!+ \! D(t) ^\top
\big(P_1(t;t)  \! +  \! P_3(t;t)  \!+ \! P_2(t)^\top  \! N(t,t)P_2(t) \big)D(t)\big]=I.
\end{align*}
Then the parameter $\th_0(\cd)$ in \eqref{s4.1-cor1-eq2} naturally vanishes and the generalized equilibrium Riccati equation \eqref{s4.1-cor1-eq1}--\eqref{s4.1-cor1-eq5} reduces to an equilibrium Riccati equation \eqref{s4.1-cor1-eq1}--\eqref{s4.1-cor1-eq2} without parameter $\th_0(\cd)$.
\end{remark}

Since the seminal work \cite{Peng-1993} on optimal control problems for forward-backward stochastic differential equations (FBSDEs), this area of research has received significant attention. Numerous studies have been conducted, although it is impossible to provide an exhaustive list. Interested readers can refer to \cite{Hu-2018,Ma-Yong-1999,Wang-Wu-Xiong-2013,Yong-2010}, as well as the references cited therein.

It is worth noting that in the aforementioned literature, the essential time-inconsistency of optimal control problems for FBSDEs has been ignored, with the focus being on pre-committed optimal control. However, a recent study \cite{H.Wang-2022} revealed that time-inconsistency is indeed present in these problems. This implies that the pre-committed optimal control obtained at the initial time may not remain optimal at a later time. 

The study of time-inconsistent problems can be traced back to the mid-18th century, with works by Hume \cite{Hume-1739} and Smith \cite{Smith-1759} analyzing the time-inconsistent behavior of animals and humans. Strotz \cite{Strotz-1955} was the first to mathematically formulate a time-inconsistent problem within an economic context. Since then, time-inconsistent problems have been widely studied in economics and finance (e.g., \cite{Basak-2008,Ekeland-Pirvu-2008,Goldman-1980,Pelege-Yaari-1973,Pollak}), focusing mainly on discrete dynamic systems, simple ordinary differential equations (ODEs), or stochastic differential equations (SDEs). Inspired by these studies, researchers in control theory started investigating time-inconsistent optimal control problems. Substantial literature has been published on this subject, including \cite{Dou-2020,Hu-2012,Hu-2017,Lu-2023,Ma-2023,H.Wang-2022,Yong-2014,Yong-2017}, as well as the references cited therein, for time-inconsistent linear-quadratic (LQ) problems for SDEs. Among these works, there are two main methods employed. The first one is based on \textit{multiperson differential games}. By this method, Yong \cite{Yong-2017} obtained a unique  closed-loop equilibrium strategy for time-inconsistent stochastic LQ problems under \textit{standard conditions} (namely, the control cost weighting matrix $R$ is uniformly positive definite, and the other weighting coefficients $Q,G_1$ are positive semidefinite). The second one is based on variational approaches  and decoupling techniques. Hu et al. \cite{Hu-2012,Hu-2017} applied this method to  obtain a unique equilibrium control and provided an equivalent characterization of that control for time-inconsistent stochastic LQ problems under \textit{standard conditions} and a singular case where they still require $R+D^\top MD$  to be uniformly positive definite (Here  $M$ satisfies a certain Riccati equation).   In these works, the proof of the existence of  closed-loop equilibrium strategies or equilibrium controls essentially requires $R+D^\top P_1D$ or $R+D^\top MD$ to be uniformly positive definite. Recently, through multiperson differential games and the establishment of a sharp estimate, L\"u and Ma \cite{Lu-2023} introduced  conditions weaker than prior researches, where $R+D^\top P_1D$ can be singular, to guarantee the existence of closed-loop equilibrium strategies for time-inconsistent stochastic LQ problems.  It should be noted that in the aforementioned works, the control systems are SDEs, and the time-inconsistency arises from \textit{ time-preferences} or \textit{risk-preferences}.  Very recently, Wang et al. \cite{H.Wang-2022} discovered that  the forward-backward structure of the controlled system can also lead to time-inconsistency. They introduced a general framework for time-inconsistent optimal control problems for FBSDEs, and proved the well-posedness of the \textit{ equilibrium Hamilton-Jacobi-Bellman equation} when the diffusion
coefficient does not contain the control variable. They also posed a time-inconsistent FBSDE linear-quadratic (TI-FBSLQ) problem but did not provide a proof for the well-posedness of the equilibrium Riccati equation.

Despite the rich literature on optimal control theory for FBSDEs and time-inconsistent optimal control theory for SDEs, there has been limited research on time-inconsistent optimal control theory for FBSDEs, with the exception of the aforementioned study in \cite{H.Wang-2022}. 

The main contributions of the current paper are as follows:

\begin{enumerate}
\item We obtained the generalized equilibrium Riccati equation for Problem (TI-FBSLQ), which is a coupled equilibrium Riccati equation system with parameters and constraint conditions. The solvability of this equation fully characterizes the existence of closed-loop equilibrium strategies for Problem (TI-FBSLQ). We also provide an equivalent characterization of the closed-loop equilibrium strategy. To obtain the results,  the main difficulties lie in three aspects:  First,  the cost weighting matrices for the state and the control are allowed  to be indefinite. In \cite{Lu-2023}, L\"u and Ma  utilized multiperson differential games to relax the standard conditions to some  extent,   and derived a family of equilibrium Riccati equations depending on parameters  to obtain   closed-loop  equilibrium strategies.  However, the method  in \cite{Lu-2023} is not suitable for handling  the more general indefinite case. To address the indefinite case,  we adopt a different approach,  the variational approach. Second, the estimate for the martingale term $Z$ in the BSDE is generally not enough for us to handle the problem. To overcome this difficulty, we notice that  the BSDE   is actually always coupled with a nice SDE, which may imply that  $Z$ have some good properties. Hence we  technically  use decoupling techniques multiple times  to dig the hidden information of $Z$. Eventually, we   introduce  fourteen auxiliary equations $\{\cP_{i}\}_{i=1}^{14}$ for decoupling. Third, the time-preferences, namely, $Q(\cd,t),R(\cd,t),M(\cd,t),N(\cd,t),G_1(t),G_2(t)$ depending on initial time $t$,  pose another obstacle to obtaining the result.  We overcome it by establishing  stability estimates for some auxiliary  equations generated during the derivation and subsequently obtaining  some convergence results.

\item Our study reveals a new and interesting phenomenon. For Problem (TI-FBSLQ), even if closed-loop equilibrium strategies exist, the solvability of the generalized equilibrium Riccati equation depends on the choice of the parameter. This is different from the time-consistent scenario, where the existence of closed-loop optimal control implies the solvability of the generalized Riccati equation.

\item When the state is one-dimensional and under slightly stronger assumptions than standard conditions, we prove the well-posedness of the generalized equilibrium Riccati equation for any chosen parameter. From this result, we can explicitly obtain a closed-loop equilibrium strategy for Problem (TI-FBSLQ). The main difficulty lies in the complicated nature of the coupled differential equation system with non-local terms and constraint conditions. We introduce assumptions to eliminate the constraint conditions and simplify the problem into a system of differential equations with non-local terms. By appropriately applying the Banach fixed-point theorem piecewise, we establish the well-posedness on the entire interval.
\end{enumerate}

The rest of this paper is organized as follows.  Section \ref{sec-proof-pri} is devoted to the proof of Theorem \ref{s4.1-cor1} and Section \ref{sec-proof-main}  addresses the proof of Theorem \ref{th2}.

\section{Proof of Theorem \ref{s4.1-cor1}}\label{sec-proof-pri}

The purpose of this section is to provide a proof for Theorem \ref{s4.1-cor1}. To achieve this, it is crucial to demonstrate the following equivalent characterization of closed-loop equilibrium strategies for Problem (TI-FBSLQ).
\begin{theorem}\label{s4.1-th1}
Let Assumption \ref{s4.1-H1} hold. Then $\Theta(\cd)\in L^2(0,T;\mathbb{R}^{k\times n})$ is a  closed-loop equilibrium strategy of Problem  (TI-FBSLQ) if and only if for  $\ae t\in[0,T],$
\begin{equation}\label{s4.1-th1-eq1}
	\begin{array}{ll}
		R(t,t) +  D(t)^\top {\bf P_1}(t;t) D(t)  +  D(t)^\top {\bf P_3}(t;t) D(t)     +  D(t)^\top {\bf P_2}(t)^\top N(t,t)	{\bf P_2}(t)  D(t)  \geq 0,
	\end{array}
\end{equation}
and
\begin{equation}\label{s4.1-th1-eq2}
\begin{aligned}
&B(t)^\top {\bf P_1}(t;t)+ D (t)^\top {\bf P_1}(t;t)C_\Theta(t)  + {B}(t)^\top {\bf P_3}(t;t)   +  D(t)^\top {\bf P_3}(t;t)C_\Theta(t) \\&+R(t,t)\Theta(t)  +\big(\widehat{B}(t)^\top+ {B}(t)^\top {\bf P_1}(t)^\top + D (t)^\top {\bf P_1}(t)^\top  \widehat{D}(t)^\top\big ) G_2(t){\bf P_1}(t) \\&+   D(t)^\top  {\bf P_1}(t)^\top N(t,t){\bf P_1}(t)C_\Theta(t)    =0,
\end{aligned}
\end{equation}
where $({\bf P_1}(\cd;\cd),{\bf P_2}(\cd),{\bf P_3}(\cd;\cd))$ satisfies the following equations:
\ss
\begin{equation}\label{s4.1-th1-eq3}
	\!\!	\begin{cases}\ds
		\frac{d{\bf P_1}(s;t)}{ds}\!+\! {\bf P_1}(s;t) A_\Theta \! +\! A_\Theta^\top  {\bf P_1}(s;t) \!  \\ \ns\ds \q+ C_\Theta^\top  {\bf P_1}(s;t)C_\Theta +Q(s,t) \!+\!\Th^\top R(s,t)\Th=\!0, & 0\leq t\leq s\leq T,\\
		\ns\ds \frac{d{\bf P_2}}{ds}+ {\bf P_2}A_\Theta+  \widehat{A}_\Theta  +\widehat{C}{\bf P_2} +\widehat{D} {\bf P_2}C_\Theta =0, & 0\leq t\leq s\leq T,\\
		\ns\ds \frac{ d{\bf P_3}(s;t)}{ds} + {\bf P_3}(s;t)A_\Theta 	+A_\Theta ^\top {\bf P_3}(s;t)   +C_\Theta ^\top {\bf P_3}(s;t)C_\Theta \\ \ns\ds \q    +{\bf P_2}^\top M(s,t){\bf P_2}  + C_\Theta ^\top {\bf P_2}^\top N(s,t) 	{\bf P_2} C_\Theta  =0, &  0\leq t\leq s\leq T,\\
		\ns\ds {\bf P_1}(T;t)=G_1(t),	\q {\bf P_2}(T)=H, \q 	{\bf P_3}(T;t)=0, & 0\leq t\leq T.
	\end{cases}
\end{equation}
\end{theorem}

\begin{proof}[Proof of Theorem \ref{s4.1-th1}] 
Since the proof is long, we divide it  into several steps.  

\ms

{\bf Step 1}. In this step, we compute $\cJ(t,x;u^\e(\cd))-\cJ(t,x;{u(\cd)})$. 

\ms

For any $(t,x)\in [0,T)\times L_{\mathcal{F}_t}^2(\Omega;\mathbb{R}^n)$,
$v\in L_{\mathcal{F}_t}^2(\Omega;\mathbb{R}^k)$ and $\Theta(\cd)\in L^2(0,T;\mathbb{R}^{k\times n})$, consider the following equations:
\begin{equation}\label{s4.2-XY}
\begin{cases}\ds
	dX=A_\Theta Xds+C_\Theta XdW(s),  & s\in [t,T],\\
	\ns\ds dY=-\big(\widehat{A}_\Theta X+\widehat{C}Y+\widehat{D}Z \big)ds+Z dW(s),  & s\in [t,T],\\
	\ns\ds X(t)=x,\q Y(T)=HX(T),
\end{cases}
\end{equation}
and 
\begin{equation}\label{s4.2-XY-1}
\begin{cases}\ds
	dX^\e(s;t)=\big(A_\Theta X^\e(s;t)+\chi_{[t,t+\e]}B v\big)ds +\big(C_\Theta X^\e(s;t)+\chi_{[t,t+\e]}D v \big)dW(s),  & s\in [t,T],\\
	\ns\ds dY^\e(s;t)=-\big(\widehat{A}_\Theta X^\e(s;t)   + \!\chi_{[t,t+\e]}\widehat{B} v  + \widehat{C} Y^\e(s;t) + \widehat{D} Z^\e(s;t)\big)ds  \\
	\ns\ds\qq\qq\q +Z^\e(s;t)dW(s),  & s\in [t,T],\\
	\ns\ds X^\e(t;t)=x,\q Y^\e(T;t)=HX^\e(T;t).
\end{cases}
\end{equation}
Let $X^\e_0:=X^\e-X,Y^\e_0:=Y^\e-Y,Z^\e_0:=Z^\e-Z$. Then we have
\begin{equation}\label{var_eq}
\begin{cases}\ds 
	dX^\e_0(s;t)=\big(A_\Theta  X^\e_0(s;t)+\chi_{[t,t+\e]} B v \big)ds +\big( C_\Theta  X^\e_0(s;t) + \chi_{[t,t+\e]} D v\big)dW(s), & s\in [t,T],\\
	\ns\ds dY^\e_0(s;t)= -\big( \widehat{A}_\Theta  X^\e_0(s;t) + \chi_{[t,t+\e]}\widehat{B} v + \widehat{C} Y^\e_0(s;t) + \widehat{D} Z^\e_0(s;t) \big)ds \\
	\ns\ds\qq \qq\q+ Z^\e_0(s;t)dW(s), & s\in [t,T],
	\\
	\ns\ds X^\e_0(t;t)=0,\q Y^\e_0(T;t)=HX^\e_0(T;t).
\end{cases}
\end{equation}

A	direct calculation gives that
\begin{align*}
&	\cJ(t,x;u^\e(\cd))-\cJ(t,x;{u(\cd)})\\   &=\frac{1}{2}\mE_t\bigg[  \int_{t}^{T}\Big( \big\langle Q(s,t)X^\e_0(s;t),X^\e_0(s;t)\big \rangle + \big \langle R(s,t)\big(\Theta X^\e_0(s;t)+v\chi_{[t,t+\e]}\big),\Theta X^\e_0(s;t) \\ & \qq\qq\qq +v\chi_{[t,t+\e]} \big \rangle   + \big\langle M(s,t)Y^\e_0(s;t),Y^\e_0(s;t) \big\rangle + \big \langle N(s,t)Z^\e_0(s;t),Z^\e_0 (s;t)\big\rangle  \Big)ds  \\&\qq\q   +2  \int_{t}^{T}\Big( \big\langle Q(s,t)X,X^\e_0 (s;t) \big\rangle +  \big\langle R(s,t)\Theta X,\Theta X^\e_0(s;t) +v\chi_{[t,t+\e]} \big\rangle   \\& \qq\qq\q  +\big \langle M(s,t)Y,Y^\e_0(s;t) \big \rangle    +\big\langle N(s,t)Z,Z^\e_0 (s;t)\big \rangle  \Big)ds +   \big\langle G_1(t)X^\e_0(T;t),X^\e_0(T;t)\big \rangle \\& \qq\q +\big \langle G_2(t)Y^\e_0(t;t),Y^\e_0(t;t) \big\rangle  + 2   \big\langle G_1(t)X(T),X^\e_0(T;t) \big\rangle + 2\big\langle G_2(t)Y(t),Y^\e_0(t;t)\big \rangle   \bigg].
\end{align*}
Consequently,
it holds that
\begin{equation}\label{11.28-eq2}
\begin{array}{ll}\ds
	\cJ(t,x;u^\e(\cd))-\cJ(t,x;{u(\cd)})\\
	\ns\ds = J_1(t,x)+J_2(t,x)+J_3(t,x)+J_4(t,x)+\mE_t\int_{t}^{t+\e}\big\langle \Theta^\top R(s,t)v,X^\e_0 (s;t) \big  \rangle ds, 
\end{array}
\end{equation}
where\ss 
\begin{equation}\label{11.28-eq3}
\begin{cases} \ds 
	J_1(t,x):=\mE_t\int_{t}^{T} \big( \big\langle F_1 (s,t), X^\e_0(s;t) \big\rangle+ \big\langle F_2(s,t),v\chi_{[t,t+\e]}\big \rangle   \big)ds\\ \ns\ds \hspace{4.5em}+ \mE_t \big \langle G_1(t)X(T),X^\e_0(T;t) \big \rangle,\\ \ns \ds 
	J_2(t,x):=\frac{1}{2} \mE_t \int_{t}^{T} \big\langle F_1^\e(s,t),X^\e_0 (s;t)\big\rangle ds+ \frac{1}{2}\mE_t \big\langle G_1(t) X^\e_0(T;t),X^\e_0(T;t)\big \rangle ,\\
	\ns \ds
	J_3(t,x):=\frac{1}{2}\mE_t \int_{t}^{T}\big( \big\langle F^\e_3 (s,t),Y^\e_0 (s;t)\big\rangle+ \big \langle F^\e_4(s,t) ,Z^\e_0(s;t) \big\rangle  \big)ds \\ \ns\ds \hspace{4.5em} + \frac{1}{2}\mE_t \big\langle G_2(t)Y^\e_0(t;t),Y^\e_0(t;t) \big \rangle,\\
	\ns \ds
	J_4(t,x):= \mE_t \int_{t}^{T}\big(\big \langle F_3 (s,t),Y^\e_0 (s;t)\big\rangle+  \big\langle F_4 (s,t),Z^\e_0 (s;t)\big\rangle  \big)ds\\ \ns\ds \hspace{4.5em}+ \mE_t\big \langle G_2(t)Y(t),Y^\e_0(t;t) \big \rangle,
\end{cases}
\end{equation}
with 
\begin{equation}\label{11.28-eq4}
\begin{cases}\ds
	F_1(s,t)\equiv \big(Q(s,t)+\Theta(s)^\top R(s,t)\Theta(s) \big) X(s),\\ 
	\ns \ds F_2(s,t)\equiv \frac{1}{2}R(s,t)v+R(s,t)\Theta(s) X(s),\\
	\ns \ds F_1^\e(s,t)\equiv  \big( Q(s,t)+\Theta(s)^\top R(s,t)\Theta(s)\big) X^\e_0(s;t),\\
	\ns \ds F_3(s,t)\equiv M(s,t)Y(s),\q F_4(s,t)\equiv N(s,t)Z(s),\\
	\ns \ds F^\e_3(s,t)\equiv M(s,t)Y^\e_0(s;t),\q F^\e_4(s,t)\equiv N(s,t)Z^\e_0(s;t).
\end{cases}
\end{equation}
%



\ms

{\bf Step 2}. In this step, we prove the following estimate:
\begin{equation}\label{s4.2-cor1}
\begin{aligned}
	&\mE_t\sup_{s\in[t,T]}|X^\e_0(s;t)|_{\dbR^n}^2\leq \mathcal{C}\e|v|_{\dbR^k}^2,\q \hbox{\rm a.s.,}\\
	&\mE_t\Big(\sup_{s\in[t,T]}|Y^\e_0(s;t)|_{\dbR^m}^2 + \int_{t}^{T}|Z^\e_0(s;t)|_{\dbR^m}^2 ds\Big)\leq  \mathcal{C}\e|v|_{\dbR^k}^2,\q \as
\end{aligned}
\end{equation} 

The inequality \eqref{s4.2-cor1} should be a known result. However, we failed to find an exact reference for it. Hence, we provide a proof here for the convenience of readers.

First,  for any $\cA\in \cF_t$, we have
$$
\begin{cases}\ns\ds
	d\chi_{\cA}X^\e_0(s;t)= \big(\chi_{\cA} A_\Theta X^\e_0(s;t) + \chi_{A}\chi_{[t,t+\e]} Bv \big)ds \\\ns\ds  \hspace{6.5em} + \big(\chi_{\cA}  C_\Theta X^\e_0(s;t) + \chi_{\cA} \chi_{[t,t+\e]}Dv\big)dW(s),\q s\in [t,T],\\ \ns\ds
	\chi_{\cA}X^\e_0(t;t)=0.
\end{cases}
$$
Then,
\begin{align}\label{pr3.1-pf-eq1} 
&	\mE\Big(\chi_{\cA}  \sup_{s\in[t,T]} |X^\e_0(s;t)|^2_{\dbR^n} \Big)\nonumber \\ 
	&=  \mE\Big(  \sup_{s\in[t,T]} |\chi_{\cA}  X^\e_0(s;t)|^2_{\dbR^n} \Big) \nonumber\\& \leq \cC \mE \Big[  \Big( \int_{t}^{t+\e}|\chi_{\cA} B  v  |_{\dbR^n} ds\Big)^2 + \int_{t}^{t+\e} |\chi_{\cA}D  v|^2_{\dbR^n}ds \Big]\\&\leq  \cC \mE \big(\chi_{\cA} \e  |v|^2_{\dbR^k} \big),\nonumber
\end{align}
where the constant $\cC$ only depends on $A_\Theta$, $C_{\Th}$, $B$, $D$, $T$ and is independent of $\cA$.
Since $v$ is $\cF_t$-measurable, we have
$$
\mE_t\sup_{s\in[t,T]}|X^\e_0(s;t)|_{\dbR^n}^2\leq \mathcal{C}\e|v|_{\dbR^k}^2,\q \hbox{\rm a.s.}
$$

Secondly, we similarly consider the following equation:
$$
\begin{cases}\ns\ds
	d\chi_{\cA}Y^\e_0(s;t)=-\big(\chi_{\cA} \widehat{A}_\Th X^\e_0(s;t)+ \chi_{\cA}\chi_{[t,t+\e]}\widehat{B}v + \chi_{\cA} \widehat{C}Y^\e_0 (s;t) \\\ns\ds\hspace{7.5em}  +  \chi_{\cA}\widehat{D} Z^\e_0(s;t) \big)ds + \chi_{\cA} Z^\e_0 (s;t)dW(s),\q s\in [t,T],\\ \ns\ds
	\chi_{\cA}Y^\e_0(T;t)=\chi_{\cA}H X^\e_0(T;t).
\end{cases}
$$
Then from \eqref{pr3.1-pf-eq1}, we have
\begin{align*}
	&\mE \bigg[\chi_{\cA} \Big( \sup_{s\in[t,T]}|Y^\e_0(s;t)|^2_{\dbR^m} +\int_{t}^{T}|Z^\e_0(s;t)|^2_{\dbR^m} ds \Big) \bigg] \\& =  \mE \  \Big( \sup_{s\in[t,T]}|\chi_{\cA} Y^\e_0(s;t)|^2_{\dbR^m} +\int_{t}^{T}|\chi_{\cA}  Z^\e_0(s;t)|^2_{\dbR^m} ds \Big)\\&\leq  \cC \mE \bigg[\chi_{\cA}  |X^\e_0(T;t)|^2_{\dbR^n}+   \Big( \int_{t}^{T} \big  | \chi_{\cA}\widehat{A}_\Theta    X^\e_0(s;t)+ \chi_{\cA}\chi_{[t,t+\e]}\widehat{B}  v  \big|_{\dbR^m}  ds\Big)^2  \bigg]  \\&\leq \cC \mE (\chi_{\cA} \e |v|_{\dbR^k}^2 ).
\end{align*}
This implies that
$$
\mE_t\bigg(\sup_{s\in[t,T]}|Y^\e_0(s;t)|_{\dbR^m}^2 + \int_{t}^{T}|Z^\e_0(s;t)|_{\dbR^m}^2 ds\bigg)\leq  \mathcal{C}\e|v|_{\dbR^k}^2,\q \as
$$

In the following four steps, we deal with $J_i(t,x)$, $i=1,\cds,4$.

\ms

{\bf Step 3}. In this step, we give an estimate of  $J_1(t,x)$.

\ms

First, consider the following adjoint equation:
\begin{equation}\label{s4.2.1-eq3}
\begin{cases}
dY_1(s;t) =-\big[ A_\Theta^\top Y_1 (s;t)+C_\Theta^\top Z_1(s;t) + \big( Q(s;t) \\\ns\ds \hspace{6em} +\Theta^\top R(s;t) \Theta  \big) X \big] ds+Z_1(s;t)dW(s),\q s\in [t,T],\\ \ns \ds 
Y_1(T;t)=G_1(t)X(T).
\end{cases}
\end{equation} 
Applying Ito's  formula for $\langle Y_1,X^\e_0 \rangle$, we have
\begin{equation} \label{s4.2-P1-1}
J_1(t,x)=\mE_t\int_{t}^{t+\e} \big\langle B^\top Y_1(s;t)+D^\top Z_1(s;t)+F_2(s,t),v \big \rangle ds.
\end{equation}

Next, we get rid of the terms containing $Y_1$ and $Z_1$ in \eqref{s4.2-P1-1}. To this end, we introduce the following equation:
\begin{equation} \label{s4.2-P1}
\begin{cases}\ns \ds
\frac{d\cP_1(s;t)}{ds}+ \cP_1(s;t) A_\Theta    +A_\Theta  ^\top  \cP_1(s;t) \\ \ns\ds \q   +C_\Theta  ^\top  \cP_1(s;t)C_\Theta  +Q(s,t)+\Theta  ^\top R(s,t)\Theta   =0, & s\in [t,T],\\\ns \ds 
\cP_1(T;t)=G_1(t).
\end{cases}
\end{equation}
By \eqref{s4.2-XY} and \eqref{s4.2-P1},  we have
\begin{equation*}
\begin{aligned}
d\big(\cP_{1}X\big)&=-\big(\cP_1 A_\Theta  +A_\Theta^\top  \cP_1   +C_\Theta^\top  \cP_1 C_\Theta+Q +\Theta^\top R\Theta  \big)Xds +\cP_1 A_\Theta Xds+\cP_1C_\Theta XdW(s)
\\&=-\big(A_\Theta^\top  \cP_1 + C_\Theta^\top  \cP_1 C_\Theta  +Q +\Theta^\top R\Theta \big )X   ds+  \cP_1C_\Theta X dW(s).
\end{aligned}
\end{equation*}
This, together with $\cP_1(T;t)=G_1(t)$, implies that $(\cP_1X,\cP_1C_\Theta X)$ is a solution to \eqref{s4.2.1-eq3}. By the uniqueness of the solution to \eqref{s4.2.1-eq3}, we obtain
\begin{equation}\label{s4.2.1-eq2}
\begin{cases}
\ds Y_1(\cd;t)=\cP_1(\cd;t)X(\cd),\\ 
\ns\ds	Z_1(\cd;t)=\cP_1(\cd;t)C_\Theta (\cd)X(\cd).
\end{cases}
\end{equation}
This, together with \eqref{s4.2-P1-1}, implies that
\begin{equation}\label{11.8-eq5}
\begin{aligned}
&J_1(t,x)\\&=\mE_t\int_{t}^{t+\e} \big\langle B^\top Y_1(s;t)+D^\top Z_1(s;t)+F_2(s,t),v \big\rangle ds\\&=
\mE_t\int_{t}^{t+\e} \Big \langle \big( B^\top \cP_1(s;t)+ D^\top \cP_1(s;t)C_\Theta+R(s,t)\Theta\big)X+\frac{1}{2}R(s,t)v,v \Big  \rangle ds
\end{aligned}
\end{equation}

\ms

{\bf Step 4}. In this step, we give an estimate of  $J_2(t,x)$.

\ms

First, we introduce the following adjoint equation:
\begin{equation}\label{11.20-eq1}
\begin{cases}\ds
dY_2(s;t)=-\big[A_\Theta^\top Y_2(s;t)+C_\Theta^\top Z_2(s;t) +\big(Q(s,t)+\Theta^\top R(s,t)\Theta\big) X^\e_0(s;t)\big]ds  \\ \ns\ds \hspace{6em} +Z_2(s;t)dW(s), & s\in [t,T],\\ \ns \ds 
Y_2(T;t)=G_1(t)X^\e_0(T;t).
\end{cases}
\end{equation} 
Applying Ito's formula for $\langle Y_2,X^\e_0 \rangle$, we have
\begin{equation}\label{11.20-eq2}
J_2(t,x)=\frac{1}{2}\mE_t\int_{t}^{t+\e} \big  \langle B^\top Y_2 (s;t)+D^\top Z_2(s;t),v \big \rangle ds.
\end{equation}

To get rid of the terms containing $Y_2$ and $Z_2$ in \eqref{11.20-eq2}, we introduce the following equations: 
\begin{equation}\label{s4.2-P2}
\begin{cases} \ns\ds 
\frac{d\cP_2(s;t)}{ds}+ \cP_2(s;t)A_\Theta      +A_\Theta  ^\top \cP_2(s;t) + C_\Theta  ^\top \cP_2(s;t)  C_\Theta   + Q(s,t)+\Theta  ^\top R(s,t)\Theta   =0, & s\in [t,T],  \\ \ns\ds 
\cP_2(T;t)=G_1(t), 
\end{cases}
\end{equation}
and
\begin{equation}\label{s4.2-P3}
\begin{cases} \ds 
\frac{d \cP_3(s;t)}{ds}+ A_\Theta   ^\top \cP_3(s;t)+ \chi_{[t,t+\e]}\cP_2(s;t) B + \chi_{[t,t+\e]}C_\Theta   ^\top \cP_2(s;t) D   =0, & s\in [t,T],\\ \ns\ds
\cP_3(T;t)=0.
\end{cases}
\end{equation}
From \eqref{var_eq}, \eqref{s4.2-P2} and \eqref{s4.2-P3},  we get
\begin{equation*}
\begin{array}{ll}\ds
d\big(\cP_2 X^\e_0 + \cP_3v \big) \\ 
\ns\ds=-\big(\cP_2 A_\Theta    +A_\Theta^\top \cP_2  +C_\Theta^\top \cP_2  C_\Theta + Q +\Theta^\top R \Theta\big)X^\e_0ds \\
\ns \ds \q - \big(A_\Theta ^\top \cP_3 v +  \chi_{[t,t+\e]}\cP_2 Bv +   \chi_{[t,t+\e]}C_\Theta ^\top \cP_2  D v\big)ds \\\ns \ds \q +\cP_2\big( A_\Theta X^\e_0+\chi_{[t,t+\e]}Bv \big)ds  
  +\cP_2 \big(C_\Theta X^\e_0+\chi_{[t,t+\e]} Dv \big)dW(s)  \\
\ns \ds = -\big [ A_\Theta^\top\big ( \cP_2 X^\e_0  + \cP_3 v\big) +C_\Theta^\top \cP_2 \big( C_\Theta X^\e_0 + \chi_{[t,t+\e]}Dv\big) + \big (Q +\Theta^\top R \Theta) X^\e_0  \big ]ds \\ \ns \ds \qq  +   \cP_2\big( C_\Theta X^\e_0 + \chi_{[t,t+\e]} Dv  \big) dW(s).
\end{array}
\end{equation*}
This, together with $\cP_2(T;t)=G_1(t)$ and $\cP_3(T;t)=0$, implies that $\big(\cP_2 X^\e_0 + \cP_3v, \cP_2 (C_\Theta X^\e_0 + \chi_{[t,t+\e]} Dv) \big)$ is a solution to \eqref{11.20-eq1}. Then,  the uniqueness of the solution to \eqref{11.20-eq1}  implies that 
\begin{equation}\label{s4.2.2-eq1}
	\begin{cases}\ds
		Y_2(\cd;t) =\cP_2(\cd;t) X^\e_0(\cd;t)+{\cP}_3(\cd;t) v,\\ \ns\ds 
		Z_2(\cd;t) = \cP_2 (\cd;t)\big(  C_\Theta (\cd)X^\e_0(\cd;t)+D(\cd)v\chi_{[t,t+\e]}  \big).
	\end{cases}
\end{equation}

By Gronwall's inequality, we get from \eqref{s4.2-P3} that
\begin{equation}\label{s4.2.2-eq2}
\begin{aligned}
\sup_{s\in[t,T]}|\cP_3(s;t)|_\infty^2 &\leq \mathcal{C} \Big(\int_{t}^{t+\e} \big |\cP_2(\tau;t)B+ C_\Theta ^\top \cP_2(\tau ;t) D \big|_\infty d\tau  \Big)^2\\&\leq 
\cC \Big(\int_{t}^{t+\e} (1+|\Theta|_\infty) d\tau\Big)^2.
\end{aligned}
\end{equation}
Combining \eqref{s4.2.2-eq1}, \eqref{s4.2.2-eq2} and  \eqref{s4.2-cor1}, we obtain that
\begin{equation}\label{11.8-eq6}
\begin{aligned}
J_2(t,x)&=\frac{1}{2}\mE_t\int_{t}^{t+\e} \Big\langle B^\top Y_2(s;t)+D^\top Z_2(s;t),v\Big \rangle ds\\&=\frac{1}{2} \mE_t \!\int_{t}^{t+\e}\!\!\! \Big \langle \big (B^\top\! \cP_2(s;t)\!+ D^\top\!  \cP_2(s;t)  C_\Theta  \big)X^\e_0 (s;t)\! \\&\qq \qq \q  + \big( B^\top\! {\cP}_3(s;t)\!+\! D^\top\! \cP_2(s;t) \chi_{[t,t+\e]}D \big)v ,v \!\Big \rangle ds\\&= \mE_t\int_{t}^{t+\e} \Big \langle \frac{1}{2} D^\top \cP_2(s;t) D v ,v \Big \rangle ds+\mathcal{C} |v|_{\dbR^k}^2 o(\e).
\end{aligned}
\end{equation}

\ms

{\bf Step 5}. In this step, we give an estimate of  $J_3(t,x)$.

\ms

We first introduce the adjoint equation:
\begin{equation}\label{s4.2-Y3}
\begin{cases}\ds
dY_3(s;t)=\big(\widehat{C}^\top Y_3(s;t)+F_3^\e(s,t)\big)ds   +\big(\widehat{D}^\top Y_3(s;t)+F_4^\e(s,t)\big)dW(s),\q s\in [t,T],\\ \ns \ds 
Y_3(t;t)=G_2(t)Y^\e_0(t;t).
\end{cases}
\end{equation}
Applying Ito's formula for $\langle Y_3,Y^\e_0 \rangle$, we have that
\begin{equation*}
	\begin{array}{ll}\ns\ds 
J_3(t,x)= \frac{1}{2}\mE_t  \big\langle Y_3(T;t),HX^\e_0(T;t)  \big\rangle  +\frac{1}{2} \mE_t\int_{t}^{t+\e}  \big\langle \widehat{B}^\top Y_3(s;t),v  \big\rangle ds \\  \ns\ds  \hspace{4.5em}+\frac{1}{2}\mE_t\int_{t}^{T}  \big\langle \widehat{A}_\Theta ^\top Y_3(s;t),X^\e_0(s;t) \big \rangle ds .
\end{array}
\end{equation*}
Similar to the proof of the inequality \eqref{s4.2-cor1}, we can prove that
\begin{equation}\label{s4.2-cor2-eq1}
\mE_t\sup_{s\in[t,T]}|Y_3(s;t)|^2\leq \cC \e |v|_{\dbR^k}^2, \q \as
\end{equation}
%


Next, consider the following equation:
\begin{equation}\label{Y4}
\begin{cases}\ds
dY_4(s;t)=-\big( A_\Theta ^\top Y_4(s;t)+C_\Theta ^\top Z_4 (s;t)+ \widehat{A}_\Theta ^\top Y_3(s;t) \big)ds  +Z_4(s;t)dW(s),\q s\in [t,T],\\ \ns \ds 
Y_4(T;t)=H^\top Y_3(T;t).
\end{cases}
\end{equation}
Similar to the proof of the inequality \eqref{s4.2-cor1},  we can show that
\begin{equation}\label{s4.2-cor3-eq1}
\mE_t\Big( \sup_{s\in[t,T]}|Y_4(s;t)|^2+\int_{t}^{T}|Z_4(s;t)|^2ds \Big)\leq \cC\e |v|_{\dbR^k}^2,\q \as 
\end{equation}

Applying Ito's formula to $\langle Y_4,X^\e_0 \rangle$, we obtain
\begin{equation*}
J_3(t,x)=\frac{1}{2} \mE_t\int_{t}^{t+\e}\big \langle \widehat{B}^\top Y_3(s;t) +{B}^\top Y_4(s;t)+D^\top Z_4(s;t) ,v \big\rangle ds.
\end{equation*}

From the inequalities \eqref{s4.2-cor2-eq1} and \eqref{s4.2-cor3-eq1}, we find that
\begin{equation}\label{s4.2.3-eq1}
J_3(t,x)=\frac{1}{2}\mE_t\int_{t}^{t+\e} \big \langle D^\top Z_4(s;t),v \big \rangle ds+\mathcal{C}|v|_{\dbR^k}^2{o}(\e).
\end{equation}
In the rest of this step, we handle the term $\ds \frac{1}{2}\mE_t\int_{t}^{t+\e}\big \langle D^\top Z_4(s;t),v \big \rangle ds$.

\ss

We first introduce the following equations:
\begin{equation}\label{s4.2-P4P5}
\begin{cases}\ns\ds 
\frac{d\cP_4  }{ds}+ \cP_4  \widehat{C}  ^\top +A_\Theta  ^\top \cP_4    +C_\Theta  ^\top \cP_4   \widehat{D}  ^\top + \widehat{A}_\Theta  ^\top=0,& s\in [t,T], \\
\cP_4(T)=H^\top, 
\end{cases}
\end{equation}
and
\begin{equation}\label{s4.2-P4P5-1}
\begin{cases} 
\ds  d\cP_5(s;t)=-\big(A_\Theta  ^\top \cP_5(s;t)+C_\Theta  ^\top \cL_5(s;t)+ \cP_4 F_3^\e(s,t) + C_\Theta  ^\top \cP_4   F_4^\e (s,t) \big)ds \\ \ns\ds  \hspace{6em} +\cL_5(s;t)dW(s), & s\in [t,T],\\
\cP_5(T;t)=0.
\end{cases}
\end{equation}

Combining \eqref{s4.2-Y3}, \eqref{s4.2-P4P5} and \eqref{s4.2-P4P5-1}, we get that
\begin{equation*}
\begin{array}{ll}\ds
d\big(\cP_4 Y_3 +\cP_5 \big)\\
\ns\ds =-\big(\cP_4\widehat{C}^\top +A_\Theta^\top \cP_4  +C_\Theta^\top \cP_4 \widehat{D}^\top + \widehat{A}_\Theta^\top\big) Y_3ds  +\cP_4\big(\widehat{C}^\top Y_3+F_3^\e\big)ds  +\cP_4\big(\widehat{D}^\top Y_3+F_4^\e\big)dW(s)  \\
\ns\ds\q  - \big( A_\Theta^\top \cP_5 +C_\Theta^\top \cL_5 + \cP_4 F_3^\e  + C_\Theta^\top \cP_4 F_4^\e \big)ds+\cL_5dW(s)\\ \ns \ds  = -\big\{A_\Theta ^\top \big(\cP_4 Y_3 +\cP_5 \big) + C_\Theta ^\top \big[\cP_4 \big(\widehat{D}^\top Y_3 +F_4^\e \big) +\cL_5 \big]+ \widehat{A}_\Theta ^\top Y_3\big\}ds  +  \big[\cP_4 \big(\widehat{D}^\top Y_3 +F_4^\e \big) +\cL_5 \big]dW(s).
\end{array}
\end{equation*}
This, together with $\cP_4(T)=H^\top$ and $\cP_5(T;t)=0$, implies that $\big(\cP_4 Y_3 +\cP_5, \cP_4 (\widehat{D}^\top Y_3 +F_4^\e) +\cL_5\big)$ is a solution to \eqref{Y4}. From the uniqueness of the solution to \eqref{Y4}, we find that 
\begin{equation}\label{11.28-eq1}
	\begin{cases} \ds
Y_4(\cd;t) = \cP_4(\cd) Y_3(\cd;t) +\cP_5(\cd;t),\\ \ns\ds    Z_4(\cd;t)  =  \cP_4(\cd) \widehat{D}(\cd)^\top  Y_3(\cd;t) +\cP_4(\cd)  N(\cd,t) Z^\e_0(\cd;t) +\cL_5(\cd;t).
\end{cases} 
\end{equation}
Combining \eqref{11.28-eq1} and \eqref{s4.2-cor2-eq1}, we get that 
\begin{equation}\label{s4.2.3-eq3}
\begin{aligned}
&\frac{1}{2}	\mE_t\int_{t}^{t+\e} \big\langle D^\top Z_4(s;t),v \big\rangle ds\\&=\frac{1}{2} \mE_t\int_{t}^{t+\e} \big\langle D^\top \big(\cP_4 N(s,t)Z^\e_0(s;t) +\cL_5(s;t)\big),v \big\rangle ds+\mathcal{C}|v|_{\dbR^k}^2o(\e).
\end{aligned}
\end{equation}  

Next,  we introduce the following equations:\ss 
\begin{equation}\label{s4.2-P6P7}
\begin{cases}\ns\ds 
\frac{d\cP_6  }{ds}+ \cP_6  A_\Theta  +  \widehat{A}_\Theta    +\widehat{C}  \cP_6   +\widehat{D}   \cP_6  C_\Theta   =0, & s\in [t,T],\\
\ns\ds \frac{d \cP_7(s;t)}{ds}+\widehat{C}   \cP_7(s;t)+ \chi_{[t,t+\e]} \cP_6   B   +   \chi_{[t,t+\e]}\widehat{B} +\widehat{D}   \cP_6  D  \chi_{[t,t+\e]} =0, & s\in [t,T],\\
\ns\ds	\cP_6(T)=H,\q \cP_7(T;t)=0.
\end{cases}
\end{equation}
Similar to the proof of \eqref{11.28-eq1}, we can show that
\begin{equation}\label{s4.2.3-eq2-1}
	\begin{cases}\ds
		Y^\e_0(\cd;t) =\cP_6(\cd) X^\e_0(\cd;t) + \chi_{[t,t+\e]}\cP_7(\cd;t)v,\\ \ns\ds 	Z^\e_0(\cd;t) =\cP_6(\cd) \big( C_\Theta(\cd) X^\e_0(\cd;t) + \chi_{[t,t+\e]}D(\cd) v  \big).
	\end{cases}
\end{equation}
This, together with the inequality \eqref{s4.2-cor1}, yields
\begin{equation}\label{s4.2.3-eq2}
\begin{array}{ll}\ds
\frac{1}{2}	\mE_t\int_{t}^{t+\e} \Big\langle D^\top \cP_4 N(s,t)Z^\e_0(s;t),v \Big\rangle ds\\
\ns\ds = 	\frac{1}{2}	\mE_t\int_{t}^{t+\e} \Big\langle D^\top \cP_4 N(s,t)	\cP_6 \big( C_\Theta X^\e_0(s;t)+Dv\chi_{[t,t+\e]}  \big),v \Big\rangle ds\\ \ns\ds =\mE_t\int_{t}^{t+\e} \Big\langle 	\frac{1}{2} D^\top \cP_4 N(s,t)	\cP_6  Dv  ,v \Big\rangle ds+\mathcal{C}|v|_{\dbR^k}^2o(\e).
\end{array}
\end{equation}

\ss

Lastly, by \eqref{s4.2.3-eq2-1}, we can rewrite the equation \eqref{s4.2-P4P5} as
\begin{equation}\label{s4.2-P5}
\begin{cases}
\ds d\cP_5(s;t)=\!-\big[A_\Theta   ^\top \cP_5(s;t)+C_\Theta   ^\top \cL_5(s;t)\!+ \cP_4  M(s,t)  \big(\cP_6  X^\e_0(s;t) +  \chi_{[t,t+\e]}\cP_7(s;t)v\big)  \\ \ns\ds\qq\qq\qq + C_\Theta   ^\top \cP_4  N(s,t) 	\cP_6   \big( C_\Theta   X^\e_0(s;t)  +\chi_{[t,t+\e]} D  v\big)  \big]ds +\cL_5(s;t)dW(s),\qq  s\in [t,T],\\
\ns\ds\cP_5(T;t)=0.
\end{cases}
\end{equation}
Let us further introduce the following equations:
\begin{equation}\label{s4.2-P8P9}
\begin{cases}\ns\ds 
\frac{d\cP_8(s;t)}{ds}+ \cP_8(s;t)A_\Theta   	+A_\Theta   ^\top \cP_8(s;t)   +C_\Theta   ^\top \cP_8(s;t)C_\Theta   \\ \ns\ds \q  +\cP_4  M(s,t)\cP_6      + C_\Theta   ^\top \cP_4  N(s,t) 	\cP_6   C_\Theta   =0, &s\in[t,T], \\
\ns\ds \cP_8(T;t)=0, 
\end{cases}
\end{equation}
and
\begin{equation}\label{s4.2-P8P9-1}
	\begin{cases} \ds 
 \frac{d\cP_9(s;t)}{ds}+   A_\Theta  ^\top \cP_9(s;t) + \chi_{[t,t+\e]} C_\Theta  ^\top \cP_8(s;t) D + \chi_{[t,t+\e]} \cP_8(s;t) B   \\ \ns\ds \q   + \chi_{[t,t+\e]}\cP_4  M(s,t)  \cP_7(s;t) + \chi_{[t,t+\e]}C_\Theta  ^\top \cP_4  N(s,t) \cP_6  D  =0,  &s\in[t,T],\\
		\ns\ds  \cP_9(T;t)=0.
	\end{cases}
\end{equation}

Using Ito's formula, we have
\begin{align*}
	&d\big(\cP_8 X^\e_0 +\cP_9 v\big) \\
	\ns\ds& =-\big(\cP_8 A_\Theta 	+A_\Theta ^\top \cP_8  +C_\Theta ^\top \cP_8 C_\Theta +\cP_4 M \cP_6 + C_\Theta ^\top \cP_4 N \cP_6 C_\Theta\big) X^\e_0ds  \\
	\ns\ds &\q + \cP_8 \big( A_\Theta X^\e_0+\chi_{[t,t+\e]}Bv \big)ds +\cP_8 \big( C_\Theta X^\e_0+\chi_{[t,t+\e]} Dv \big)dW(s) \\
	\ns\ds &\q - \big(A_\Theta^\top \cP_9 + \chi_{[t,t+\e]}C_\Theta^\top \cP_8 D +\chi_{[t,t+\e]} \cP_8 B + \chi_{[t,t+\e]} \cP_4 M \cP_7   \\
	\ns\ds &\qq + \chi_{[t,t+\e]} C_\Theta^\top \cP_4 N \cP_6 D\big)v ds \\ 
	\ns\ds &= -\big[ A_\Theta ^\top \big(\cP_8 X^\e_0 +\cP_9 v\big) +C_\Theta ^\top\cP_8 \big(C_\Theta X^\e_0 +\chi_{[t,t+\e]} D v \big)   \\
	\ns\ds& \qq + \cP_4 M \big(\cP_6 X^\e_0+  \chi_{[t,t+\e]}\cP_7 v\big)+ C_\Theta^\top \cP_4 N \cP_6 \big( C_\Theta X^\e_0+\chi_{[t,t+\e]} D v\big)\big]ds \\
	\ns\ds &\q + \cP_8 \big( C_\Theta X^\e_0 + \chi_{[t,t+\e]} D v \big) dW(s).
\end{align*}
This, together with $\cP_8(T;t)=0$ and $\cP_9(T;t)=0$, implies that $\big(\cP_8 X^\e_0 +\cP_9 v, \cP_8 ( C_\Theta X^\e_0 + \chi_{[t,t+\e]} D v )\big)$ is a solution to \eqref{s4.2-P5}. Then,  the uniqueness of the solution to \eqref{s4.2-P5} implies that 
\begin{equation*}
\begin{cases} \ds
\cP_5(\cd,t) =\cP_8(\cd;t) X^\e_0(\cd;t) +\cP_9(\cd;t) v,\\ \ns\ds 	\cL_5(\cd;t) =\cP_8 (\cd;t)\big( C_\Theta(s)  X^\e_0(\cd;t) + \chi_{[t,t+\e]} D(\cd) v \big).
\end{cases}
\end{equation*}
This, together with the inequalities \eqref{s4.2-cor1}, implies that
\begin{equation}\label{s4.2.3-eq4}
\begin{aligned}
&\frac{1}{2}	\mE_t\int_{t}^{t+\e} \big \langle D^\top \cL_5(s;t),v \big \rangle ds\\&= \frac{1}{2}	\mE_t\int_{t}^{t+\e} \Big\langle D^\top \cP_8(s;t) \big( C_\Theta X^\e_0+Dv\chi_{[t,t+\e]}  \big),v \Big\rangle ds\\&=\mE_t\int_{t}^{t+\e} \Big\langle \frac{1}{2} D^\top \cP_8(s;t)  Dv  ,v \Big\rangle ds+\mathcal{C}|v|_{\dbR^k}^2o(\e).
\end{aligned}
\end{equation}

Consequently, from equations \eqref{s4.2.3-eq1},  \eqref{s4.2.3-eq3}, \eqref{s4.2.3-eq2} and \eqref{s4.2.3-eq4}, we obtain that
\begin{equation}\label{s4.2.3-eq5}
\begin{aligned}
J_3(t,x)&=\frac{1}{2}\mE_t\int_{t}^{t+\e} \big\langle D^\top Z_4(s;t),v \big\rangle ds+\mathcal{C}|v|_{\dbR^k}^2{o}(\e)\\&=\frac{1}{2}\mE_t\int_{t}^{t+\e} \Big\langle 	\big( D^\top \cP_4 N(s,t)	\cP_6  D +  D^\top \cP_8(s;t)  D \big)  v  ,v \Big\rangle ds+\mathcal{C}|v|_{\dbR^k}^2o(\e)
\end{aligned}
\end{equation}

\ms

{\bf Step 6}. In this step, we give an estimate of  $J_4(t,x)$.

\ms

We first introduce the following adjoint equation:
\begin{equation}\label{s4.2-Y5}
\begin{cases}\ds
dY_5(s;t)=\big(\widehat{C}^\top Y_5(s;t)+F_3(s,t)\big)ds +\big(\widehat{D}^\top Y_5(s;t)+F_4(s,t)\big)dW(s),\q s\in [t,T],\\ \ns \ds 
Y_5(t;t)=G_2(t)Y(t).
\end{cases}
\end{equation}
Applying Ito's formula for $\langle Y_5,Y^\e_0 \rangle$, we have
\begin{equation}\label{11.21-eq4}
	\begin{array}{ll}\ns\ds
J_4(t,x)= \mE_t  \big \langle Y_5(T;t),HX^\e_0(T;t)\big \rangle  + \mE_t\int_{t}^{t+\e} \big\langle \widehat{B}^\top Y_5(s;t),v \big\rangle ds  \\\ns\ds \qq \qq \, +\mE_t\int_{t}^{T}\big \langle \widehat{A}_\Theta^\top Y_5(s;t),X^\e_0 (s;t)\big\rangle ds .
\end{array}
\end{equation}
Similar to the treatment of $J_3(t,x)$ in Step 5, we  introduce another adjoint equation:
\begin{equation}\label{11.21-eq1}
\begin{cases} \ds
dY_6(s;t)=-\big(A_\Theta^\top Y_6(s;t)+ C_\Theta^\top Z_6(s;t) + \widehat{A}_\Theta^\top Y_5(s;t) \big)ds + Z_6(s;t)dW(s), &  s\in [t,T],\\ \ns \ds 
Y_6(T;t)=H^\top Y_5(T;t).
\end{cases}
\end{equation}
Applying Ito's formula for $\langle Y_6,X^\e_0 \rangle$, we get
\begin{equation}\label{11.21-eq2}
\begin{aligned}
J_4(t,x)=&\mE_t\int_{t}^{t+\e} \big\langle \widehat{B}^\top Y_5(s;t),v \big\rangle ds+ \mE_t\int_{t}^{t+\e} \big\langle {B}^\top Y_6(s;t)+D^\top Z_6(s;t),v\big \rangle ds\\=& \mE_t\int_{t}^{t+\e} \big\langle \widehat{B}^\top Y_5(s;t) +{B}^\top Y_6(s;t)+D^\top Z_6(s;t) ,v\big \rangle ds  .
\end{aligned}
\end{equation}

Now, our  goal is  to represent $Y_5(\cd;t)$ and $\big(Y_6(\cd;t),Z_6(\cd;t)\big)$ by $X(\cd)$.  To this end, consider the following equations:
\begin{equation}\label{s4.2-P10P11}
\begin{cases}\ns\ds 
\frac{d\cP_{10}  }{ds}+ \cP_{10}  \widehat{C}  ^\top   + A_\Theta  ^\top \cP_{10}   +C_\Theta  ^\top \cP_{10}    \widehat{D}  ^\top  + \widehat{A}_\Theta  ^\top =0,& s\in [t,T],\\ \ns\ds 
d\cP_{11}(s;t)=-\big( A_\Theta  ^\top \cP_{11}(s;t)+C_\Theta   ^\top \cL_{11}(s;t) +\cP_{10}  F_3(s,t) 	 \\\ns\ds  \hspace{6.5em}  +C_\Theta  ^\top \cP_{10}  F_4(s,t) \big)ds+\cL_{11}(s;t)dW(s), & s\in [t,T],\\ \ns\ds 
\cP_{10}(T)=H^\top,\q \cP_{11}(T;t)=0.
\end{cases}
\end{equation}
Similar to the proof of \eqref{11.28-eq1}, we can show that
\begin{equation}\label{11.21-eq3}
	\begin{cases}\ds
		Y_6(\cd;t) =\cP_{10}(\cd) Y_5(\cd;t) +\cP_{11}(\cd;t) ,\\ \ns\ds
		Z_6(\cd;t) =\cP_{10}(\cd) \big(\widehat{D}(\cd)^\top Y_5(\cd;t) +F_4(\cd,t) \big) +\cL_{11}(\cd;t).
	\end{cases}
\end{equation}
From \eqref{11.21-eq4} and \eqref{11.21-eq3}, we obtain 
\begin{align}\label{s4.2.4-eq2}
\nonumber&J_4(t,x)\\&=\mE_t\int_{t}^{t+\e}  \Big \langle \widehat{B}^\top Y_5(s;t)  +{B}^\top Y_6(s;t)  +D^\top Z_6 (s;t), v \Big\rangle ds \\&\nonumber= \mE_t\!\int_{t}^{t+\e}\! \Big\langle  \big( \widehat{B}^\top  +{B}^\top \cP_{10}  \nonumber + D^\top  \cP_{10}  \widehat{D}^\top \big)  Y_5(s;t) +{B}^\top \cP_{11}(s;t) \\& \qq \qq \q +D^\top  \cP_{10} F_4(s,t)  +D^\top \cL_{11}(s;t), v \Big\rangle ds.
\end{align}

\ms 

Next, we introduce the following equation:
\begin{equation}\label{s4.2-P12}
\begin{cases}\ns\ds 
\frac{d \cP_{12}  }{ds}+  \cP_{12}   A_\Theta   	+ \widehat{A}_\Theta    +\widehat{C}   \cP_{12}   +\widehat{D}    \cP_{12}   C_\Theta   =0,& s\in [0,T],\\ \ns\ds 
\cP_{12}(T)=H.
\end{cases}
\end{equation}
Similar to the proof of \eqref{11.28-eq1}, we can obtain that
\begin{equation}\label{s4.2-XtoYZ}
\begin{cases} \ds	Y(\cd)= \cP_{12}(\cd)X(\cd), \\
 \ns \ds Z(\cd)= \cP_{12}(\cd) C_\Theta(\cd) X(\cd). \end{cases}
\end{equation}

By \eqref{s4.2-XtoYZ} and \eqref{s4.2-Y5}, we get that
\begin{equation}\label{s4.2.4-eq1}
\begin{cases} \ds
dY_5(s;t)=\big(\widehat{C}^\top Y_5(s;t)+M(s,t)\cP_{12}X\big)ds \\ \ns\ds \hspace{5em}+\big(\widehat{D}^\top Y_5(s;t)+N(s,t)\cP_{12} C_\Theta X\big)dW(s),\q s\in [t,T],\\ \ns \ds 
Y_5(t;t)=G_2(t)\cP_{12}(t)X(t).
\end{cases}
\end{equation}
Now we study the relationship between $Y_5$ and $X$. To this end,  we introduce the following equation:\ss 
\begin{equation}\label{s4.2-P13}
\begin{cases} \ds 
	d\cP_{13}(s;t)= \big[\big( \cP_{13}(s;t)C_\Theta  -\widehat{D}  ^\top \cP_{13}(s;t)-N(s,t) \cP_{12}   C_\Theta    \big) C_\Theta   \\ \ns \ds \hspace{6em} -\cP_{13}(s;t)A_\Theta   +\widehat{C}  ^\top \cP_{13}(s;t)+M(s,t) \cP_{12}     \big]ds \\ \ns\ds  \hspace{6em}  + \big( -\cP_{13}(s;t)C_\Theta  +\widehat{D}  ^\top \cP_{13}(s;t)   +N(s,t)\cP_{12}   C_\Theta  \big) dW(s), \qq\qq s\in [t,T], \\
\ns\ds 	\cP_{13}(t;t)=G_2(t) \cP_{12}(t).
\end{cases}
\end{equation}
Similar to the proof of \eqref{11.28-eq1}, we get that 
\begin{equation}\label{s4.2.4-eq4}
Y_5(\cd;t)=\cP_{13}(\cd;t)X(\cd).
\end{equation}
Combining \eqref{s4.2.4-eq2}, \eqref{s4.2-XtoYZ} and \eqref{s4.2.4-eq4},  we obtain that
\begin{equation}\label{s4.2.4-eq3}
\begin{aligned}
	J_4(t,x)&= \mE_t\int_{t}^{t+\e} \Big\langle  \big( \widehat{B}^\top +{B}^\top \cP_{10} + D^\top  \cP_{10}  \widehat{D}^\top \big)  Y_5(s;t)  \\&\qq \qq \q +{B}^\top \cP_{11}(s;t)+D^\top  \cP_{10} F_4 (s,t)  +D^\top \cL_{11} (s;t) ,v \Big\rangle ds\\&= \mE_t\!\int_{t}^{t+\e}\! \Big\langle  \big[\big( \widehat{B}^\top +{B}^\top \cP_{10} + D^\top  \cP_{10}  \widehat{D}^\top \big)  \cP_{13}(s;t) +D^\top  \cP_{10} N(s,t)\cP_{12}C_\Theta\big] X \\ &\qq\qq\q+{B}^\top \cP_{11} (s;t)   +D^\top \cL_{11}(s;t), v \Big\rangle ds.
\end{aligned}
\end{equation}

\ss

Next, we derive the relationship between $X(\cd)$ and  $\big(\cP_{11}(\cd;t),\cL_{11}(\cd;t)\big)$.  First, 
by \eqref{s4.2-XtoYZ}, we can rewrite the second equation of \eqref{s4.2-P10P11} as
\begin{equation}\label{11.21-eq5}
\begin{cases}\ns\ds 
	d\cP_{11}(s;t)=-\big(A_\Theta   ^\top \cP_{11}(s;t)+C_\Theta   ^\top \cL_{11}(s;t) +\cP_{10}   M(s,t)\cP_{12}   X    \\ \ns\ds \hspace{6em}	 +C_\Theta   ^\top \cP_{10}   N(s,t)\cP_{12}    C_\Theta    X    \big)ds+\cL_{11}(s;t)dW(s),\\
	\cP_{11}(T;t)=0.
\end{cases}
\end{equation}

Secondly, we introduce the following equation:
\begin{equation}\label{s4.2-P14}
\begin{cases} \ds 
	\frac{d\cP_{14}(s;t)}{ds}+ \cP_{14}(s;t)A_\Theta      + A_\Theta   ^\top \cP_{14}(s;t)  +C_\Theta   ^\top \cP_{14}(s;t)C_\Theta     \\ \ns\ds \q +\cP_{10}   M(s,t)\cP_{12}      +C_\Theta   ^\top \cP_{10}   N(s,t)\cP_{12}   C_\Theta =0,& s\in [t,T],\\ \ns\ds 
	\cP_{14}(T;t)=0,
\end{cases}
\end{equation}
Similar to the proof of \eqref{11.28-eq1}, we can prove that
\begin{equation}\label{11.21-eq6}
\begin{cases} \ds \cP_{11}(\cd;t)=\cP_{14}(\cd;t)X(\cd), \\ \ns\ds 
\cL_{11}(\cd;t)=\cP_{14}(\cd;t)C_\Theta(\cd)X(\cd). 
\end{cases}
\end{equation}

From \eqref{s4.2.4-eq3} and \eqref{11.21-eq6}, we obtain that
\begin{equation}\label{s4.2-J4}
\begin{aligned}
&J_4(t,x)\\&= \mE_t\!\int_{t}^{t+\e}\! \Big\langle  \big[\big( \widehat{B}^\top \!+{B}^\top \cP_{10}\! + D^\top  \cP_{10}  \widehat{D}^\top \big)  \cP_{13}(s;t) \\&\qq\qq \q   +D^\top  \cP_{10} N(s,t)\cP_{12}C_\Theta\big] X \!+{B}^\top \cP_{11}(s;t)  \! +D^\top \cL_{11}(s;t)  ,v \Big\rangle ds\\&=
\mE_t\!\int_{t}^{t+\e}\! \Big\langle  \big[\big( \widehat{B}^\top +{B}^\top \cP_{10} + D^\top  \cP_{10}  \widehat{D}^\top \big)  \cP_{13}(s;t) \\&\qq\qq\q  +D^\top  \cP_{10} N(s,t)\cP_{12}C_\Theta\!+\! {B}^\top \cP_{14}(s;t)  + D^\top \cP_{14}(s;t)C_\Theta \big] X,v \Big\rangle ds
\end{aligned}
\end{equation}

\ms

{\bf Step 7}. We combine the results in Step 3--Step 6 to study $\cJ(t,x;u^\e(\cd))-\cJ(t,x;u(\cd))$.

\ms

From \eqref{11.28-eq2}--\eqref{11.28-eq4},  \eqref{11.8-eq5}, \eqref{11.8-eq6}, \eqref{s4.2.3-eq5} and \eqref{s4.2-J4}, for any $(t,x)\in [0,T)\times L_{\mathcal{F}_t}^2(\Omega;\mathbb{R}^n)$,
$v\in L_{\mathcal{F}_t}^2(\Omega;\mathbb{R}^k)$ and $\Theta(\cd)\in L^2(0,T;\mathbb{R}^{k\times n})$,  we have that
\begin{align}\label{s4.2-Jdelta}
\nonumber& \cJ(t,x;u^{\e}(\cd))-\cJ(t,x;{u(\cd)})\\ \nonumber&= 		\Big \langle \Big( \int_{t}^{t+\e} \frac{1}{2}R(s,t)ds + \int_{t}^{t+\e} \frac{1}{2} D^\top \cP_2(s;t) D ds  + \int_{t}^{t+\e}  \frac{1}{2} D^\top \cP_8(s;t) Dds  \\\nonumber&\q +   \int_{t}^{t+\e}  \frac{1}{2}  D^\top \cP_4 N(s,t)	\cP_6  D ds  \Big)  v,v \Big  \rangle  +  \Big \langle    \  \mE_t \int_{t}^{t+\e}  \big[   B^\top \cP_1(s;t) \\\nonumber&\q + D^\top \cP_1(s;t)C_\Theta+R(s,t)\Theta + \big( \widehat{B}^\top +{B}^\top \cP_{10}   + D^\top  \cP_{10}  \widehat{D}^\top \big)  \cP_{13}(s;t) \\&\q +D^\top  \cP_{10} N(s,t)\cP_{12}C_\Theta + {B}^\top \cP_{14}(s;t)+ D^\top \cP_{14}(s;t)C_\Theta \big] X ds  , v \Big  \rangle+ \cC |v|_{\dbR^k}^2 o(\e).
\end{align}
Moreover, by comparing the equation \eqref{s4.1-th1-eq3} with  the equations \eqref{s4.2-P1}, \eqref{s4.2-P2},   \eqref{s4.2-P4P5}, \eqref{s4.2-P6P7}, \eqref{s4.2-P8P9}, \eqref{s4.2-P10P11}, \eqref{s4.2-P12}, \eqref{s4.2-P14}, we see that
for any $\Th\in L^2(0,T;\dbR^{k\times n})$ and  $0\leq t\leq s \leq T$, it holds that
\begin{equation}\label{s4.2-pr1}
\begin{cases}\ds
{\bf P_1}(s;t)=\cP_1(s;t)= \cP_2(s;t),\\ \ns\ds {\bf P_2}(s)^\top=\cP_4(s)=\cP_{10}(s)=\cP_6(s)^\top =\cP_{12}(s)^\top,  \\ \ns\ds {\bf P_3}(s;t)=\cP_8(s;t)=\cP_{14}(s;t),
\end{cases}\q s\in [t,T].
\end{equation}

\ms

{\bf Step 8}. In this step, we establish some regularity estimate.

\ms

To this end, we first make some preparation. For $X(\cdot)$ and $Y_5(\cdot;t)=\cP_{13}(\cdot;t)X(\cd)$  (see equation \eqref{s4.2.4-eq4}), by the equations \eqref{s4.2-XY}, \eqref{s4.2.4-eq1}, we can see that $\mathbb{E}_t X(\cd;t)$ and $\mathbb{E}_t Y_5(\cdot;t)$ satisfy the following equations:
\begin{equation*}
\begin{cases}
	d\mathbb{E}_t X(s;t)=A_\Theta   \mathbb{E}_t X(s;t)ds,\qquad s\in [t,T],\\
	\mathbb{E}_t X(t;t)=x,
\end{cases}
\end{equation*}
and
\begin{equation*}
\begin{cases}
	d\mathbb{E}_t Y_5(s;t)=\big(\widehat{C}   ^\top \mathbb{E}_t Y_5(s;t)+M(s,t)\cP_{12}   \mathbb{E}_t X(s;t)\big)ds,\qquad s\in [t,T],\\
	\mathbb{E}_t Y_5(t;t)=G_2(t)\cP_{12}(t)x.
\end{cases}
\end{equation*}
Consider the following equations:
\begin{equation*}
\begin{cases}\ns\ds 
	\frac{d \Psi(s;t)}{ds}=A_\Theta    \Psi(s;t), \qquad s\in [t,T],\\ \ds 
	\Psi(t;t)=I,
\end{cases}
\end{equation*}
and
\begin{equation*}
\begin{cases}\ns\ds
	\frac{d\Gamma(s;t)}{ds}= \widehat{C}   ^\top \Gamma(s;t)+M(s,t)\cP_{12}    \Psi(s;t) , \qquad s\in [t,T],\\ \ds 
	\Gamma(t;t)=G_2(t)\cP_{12}(t).
\end{cases}
\end{equation*}
Then we have
\begin{equation}\label{s4.3-eq1}
\begin{cases}\ds
	\mathbb{E}_t X(s;t)=\Psi(s;t)x,\\
	\ns\ds
	\mathbb{E}_t Y_5(s;t)=\Gamma(s;t)x.
\end{cases}
\end{equation}

Rewrite the equation for $\Gamma(\cd;t)$ as
\begin{equation}
	\begin{cases}\ns\ds 
		\frac{d\Gamma(s;t)}{ds}= \widehat{C}   ^\top \Gamma(s;t)+M(s,t)\cP_{12}    \Psi(s;t),\q s\in [\tau,T],\\ \ds 
		\Gamma(\tau;t)=	\Gamma(\tau;t),
	\end{cases}
\end{equation}
where $ \,0\leq t\leq \tau \leq  s\leq T$. Then 
	\begin{align*}
	&	\big|\Gamma(\tau;t)\!-\!\Gamma(\tau;\tau)\big|_\infty\\
		& \leq  \big|G_2(t)\cP_{12}(t)\!-\!G_2(\tau)\cP_{12}(\tau)\big|_\infty +  \int_{t}^{\tau} \big|\widehat{C}^\top \Gamma(s;t) + M(s, t)\cP_{12} \Psi(s;t)  \big|_\infty ds   \\&\leq 
		\mathcal{C}\big(|t-\tau |+|\cP_{12}(t)-\cP_{12}(\tau)|_\infty\big ),
	\end{align*}
which, together with Assumptions \ref{s4.1-H1},  yields
\begin{equation*}
	\begin{array}{ll}\ds 
		|\Gamma(s;t)-\Gamma(s;\tau)|_\infty\\
		\ns\ds \leq |\Gamma(\tau;t)-\Gamma(\tau;\tau) |_\infty+ \int_{\tau}^{s} \big|\widehat{C}^\top \big( \Gamma(\eta;t)-\Gamma(\eta;\tau) \big)  +M(\eta,t)\cP_{12} \Psi(\eta;t) \\\ns\ds  \hspace{12.5em}-M(\eta,\tau)\cP_{12} \Psi(\eta;\tau)  \big|_\infty d\eta   \\ \ns\ds \leq 
		\mathcal{C}\big(|t-\tau |+|\cP_{12}(t)-\cP_{12}(\tau)|_\infty \big)  +\int_{\tau}^{s} \big(\mathcal{C} |\Gamma(\eta;t)-\Gamma(\eta;\tau)|_\infty+\mathcal{C} |t-\tau|\big)d\eta.
	\end{array}
\end{equation*}
This, together with Gronwall's inequality, implies 
$$
|\Gamma(s;t)-\Gamma(s;\tau)|_\infty \leq \mathcal{C}\big(|t-\tau |+|\cP_{12}(t)-\cP_{12}(\tau)|_\infty\big).
$$
By a similar (but simpler) argument, we can show that  for any $ \,0\leq t\leq \tau \leq  s\leq T$, it holds that
\begin{equation}\label{s4.3-pr1-eq1}
	\begin{cases}\ds 
		|\Psi(s;t)-\Psi(s;\tau)|_\infty \leq \mathcal{C} |t-\tau|, \\
		\ns\ds |\Gamma(s;t)-\Gamma(s;\tau)|_\infty \leq \mathcal{C}\big(|t-\tau |+|\cP_{12}(t)-\cP_{12}(\tau)|_\infty\big),\\
		\ns\ds |\cP_1(s;t)-\cP_1(s;\tau)|_\infty \leq \mathcal{C} |t-\tau|, \\
		\ns\ds 
	  |\cP_8(s;t)-\cP_8(s;\tau)|_\infty  \leq \mathcal{C} |t-\tau|. 
	\end{cases}
\end{equation}

\ms

{\bf Step 9}. In this step, we prove the following result:

\ms
\begin{equation}\label{s4.3-pr2-eq1}
\!\!\!\!\!\!\begin{cases}\ns\ds 
\lim_{\e \to 0}\frac{1}{\e} \int_{t}^{t+\e} \frac{1}{2}R(s,t)ds =\frac{1}{2}R(t,t),\q  t\in [0,T],\\  \ns\ds 
\lim_{\e \to 0}	\frac{1}{\e}\int_{t}^{t+\e} D   ^\top \cP_2(s;t) D    ds=D(t)^\top \cP_2(t;t) D(t),\q \ae t\in [0,T],\\  \ns\ds 
\lim_{\e \to 0}	\frac{1}{\e}\int_{t}^{t+\e}   D   ^\top \cP_8(s;t) D   ds= D(t)^\top \cP_8(t;t) D(t),\q \ae t\in [0,T],\\  \ns\ds 
\lim_{\e \to 0}	\frac{1}{\e}\!	\int_{t}^{t+\e} \!\!   D   ^\top \cP_4    N(s,t)	\cP_6     D    ds\!=\!  D(t)^\top \cP_4(t) N(t,t)	\cP_6(t)  D (t),\; \ae t\!\in\! [0,T],
\end{cases}
\end{equation}
and 
\begin{equation}\label{s4.3-pr2-eq2}
\begin{cases}
\ds\lim_{\e \to 0}\frac{1}{\e} 	\int_{t}^{t+\e}     B   ^\top \cP_1(s;t)\Psi(s;t)ds= B(t)^\top \cP_1(t;t),\\ \ns\ds
\lim_{\e \to 0}\frac{1}{\e} 	\int_{t}^{t+\e}  	 D   ^\top \cP_1(s;t)C_\Theta   \Psi(s;t)ds=  D(t) ^\top \cP_1(t;t)C_\Theta(t),\\ \ns\ds
\lim_{\e \to 0}\frac{1}{\e} \!\int_{t}^{t+\e}\! \big[ 	R(s,t)\Theta\!+\!{B}^\top    \cP_{14}(s;t)\big] \Psi(s;t)ds= R(t,t)\Theta(t)\!+\!B(t) ^\top \cP_{14}(t;t),\\ \ns\ds
\lim_{\e \to 0}\frac{1}{\e} \! \int_{t}^{t+\e}\!\!	D   ^\top  \cP_{10} N(s,t)\cP_{12}C_\Theta\Psi(s;t) ds= D(t) ^\top \cP_{10}(t) N(t,t)\cP_{12}(t)C_\Theta(t) ,  \\ \ns\ds
\lim_{\e \to 0}\frac{1}{\e} \int_{t}^{t+\e} D   ^\top \cP_{14}(s;t)C_\Theta   \Psi(s;t)ds= D(t) ^\top \cP_{14}(t;t)C_\Theta(t),\\ \ns\ds
\lim_{\e \to 0}\frac{1}{\e} \!	\int_{t}^{t+\e} \! \big( \widehat{B}   ^\top +{B}   ^\top \cP_{10} + D   ^\top  \cP_{10}  \widehat{D}   ^\top  \big)\Gamma(s;t)ds\\  \ns\ds = \big(\widehat{B}(t)^\top + {B}(t)^\top \cP_{10}(t) + D (t)^\top \cP_{10}(t)  \widehat{D}(t)^\top \big) G_2(t)\cP_{12}(t),
\end{cases} \q \ae t\in [0,T]
\end{equation}

Since
\begin{align*}
&\Big| \frac{1}{\e}\int_{t}^{t+\e} D   ^\top \cP_2(s;t) D    ds-D(t)^\top \cP_2(t;t) D(t)\Big|_\infty\\& \leq \frac{1}{\e}\int_{t}^{t+\e}\big|  D   ^\top \big(\cP_2(s;t)-\cP_2(t;t)\big) D    \big|_\infty ds\\&\q  +\Big|  \frac{1}{\e}\int_{t}^{t+\e}\!\!  D   ^\top \cP_2(t;t) D    ds\!-\! D(t) ^\top \cP_2(t;t) D(t) \Big|_\infty\\ &\leq \mathcal{C} \frac{1}{\e}\int_{t}^{t+\e}\big|   \cP_2(s;t)-\cP_2(t;t)  \big|_\infty ds \\&\q   +\Big|  \frac{1}{\e}\int_{t}^{t+\e}  D   ^\top \cP_2(t;t) D    ds- D(t)^\top \cP_2(t;t) D(t) \Big|_\infty\\&\leq \mathcal{C} \frac{1}{\e}\int_{t}^{t+\e}\big|   \cP_2(s;t)-\cP_2(t;t)  \big|_\infty ds+\mathcal{C} \frac{1}{\e}\int_{t}^{t+\e} | D   ^\top - D(t)^\top|_\infty ds,
\end{align*}
by  Lebesgue's differentiation theorem, we have
\begin{equation*}
\lim_{\e \to 0}	\frac{1}{\e}\int_{t}^{t+\e} D^\top \cP_2(s;t) D ds=D(t)^\top \cP_2(t;t) D(t),\q \ae t\in [0,T].
\end{equation*}
Other formulas in 	\eqref{s4.3-pr2-eq1} can be proved similarly. 

Except for the second equality in \eqref{s4.3-pr2-eq2},  the other ones can be proved similarly.  Since
\begin{align*}
& \Big |\frac{1}{\e} 	\int_{t}^{t+\e}  	 D   ^\top \cP_1(s;t)C_\Theta   \Psi(s;t)ds-  D(t)^\top \cP_1(t;t)C_\Theta(t)\Big|_\infty\\&\leq 
\frac{1}{\e} 	\int_{t}^{t+\e} \big| D    ^\top \big( \cP_1(s;t) - \cP_1(t;t)\big)C_\Theta   \Psi(s;t)\big|_\infty ds \\&\q  + \frac{1}{\e} \!	\int_{t}^{t+\e} \big| D   ^\top \cP_1(t;t)C_\Theta   \big(\Psi(s;t) - \Psi(t;t)\big)\big|_\infty ds \\&\q + \Big| \frac{1}{\e} 	\int_{t}^{t+\e}  	 D^\top \cP_1(t;t)C_\Theta\Psi(t;t)ds- D(t)^\top \cP_1(t;t)C_\Theta(t) \Big|_\infty\\&\leq \mathcal{C} \frac{1}{\e} 	\int_{t}^{t+\e} \big| \cP_1(s;t)-\cP_1(t;t)\big|_\infty(1+|\Theta   |_\infty)ds  \\&\q   + \mathcal{C} \frac{1}{\e} 	\int_{t}^{t+\e} \big|  \Psi(s;t)-\Psi(t;t)\big|_\infty(1+|\Theta   |_\infty)ds\\ &\q+ \Big| \frac{1}{\e} 	\int_{t}^{t+\e} D   ^\top \cP_1(t;t)C_\Theta   \Psi(t;t)ds- D(t)^\top \cP_1(t;t)C_\Theta(t) \Big|_\infty\\&\leq  \mathcal{C}  \Big( \frac{1}{\e} 	\int_{t}^{t+\e}  (1+|\Theta   |_\infty)^2ds \Big)^{1/2} \bigg[ \Big( \frac{1}{\e} 	\int_{t}^{t+\e} \big| \cP_1(s;t)-\cP_1(t;t)\big|_\infty^2ds\Big)^{1/2} \\&\q  +\Big( \frac{1}{\e} 	\int_{t}^{t+\e} \big| \Psi(s;t)-\Psi(t;t)\big|_\infty^2ds\Big)^{1/2} \bigg] \\&\q  +\Big| \frac{1}{\e} 	\int_{t}^{t+\e}  	 D^\top \cP_1(t;t)C_\Theta\Psi(t;t)ds - D(t) ^\top \cP_1(t;t)C_\Theta(t) \Big|_\infty
\\&\leq  \mathcal{C}  \Big( \frac{1}{\e} 	\int_{t}^{t+\e}   (1\! +\! |\Theta   |_\infty)^2ds \Big)^{1/2} \! \bigg[ \Big(  \frac{1}{\e} 	\int_{t}^{t+\e}  \big| \cP_1(s;t)\! -\! \cP_1(t;t)\big|_\infty^2ds\! \Big)^{1/2}   \\&\q +\!  \Big(  \frac{1}{\e} \! 	\int_{t}^{t+\e} \! \! \! \big| \Psi(s;t)\! -\! \Psi(t;t)\big|_\infty^2ds\Big)^{1/2} + \Big( \frac{1}{\e} 	\int_{t}^{t+\e} \big| D   ^\top-D(t)^\top\big|_\infty^2ds \Big)^{1/2} \bigg] \\&\q +\mathcal{C} \frac{1}{\e} 	\int_{t}^{t+\e} \big|  C_\Theta   -C_\Theta(t) \big|_\infty ds,
\end{align*}
we get the second equality in \eqref{s4.3-pr2-eq2} immediately. 

\ss

{\bf Step 10}. In this step, we prove the ``only if" part of Theorem \ref{s4.1-th1}.

\ss

Choose any  sequence $\{\e_j\}_{j=1}^\infty\subset (0,+\infty)$ satisfying  $\lim\limits_{j\to\infty}\e_j=0$. By \eqref{s4.2-Jdelta}, 
for any $ (t,x,v)\in [0,T)\times L_{\mathcal{F}_t}^2(\Omega;\mathbb{R}^n)\times L_{\mathcal{F}_t}^2(\Omega;\mathbb{R}^k)$, we have that
\begin{align*}
& \liminf_{j\to \infty}	\frac{\cJ(t,x;u^{\e_j}(\cd))-\cJ(t,x;{u(\cd)})}{\e_j}\\&= 	\liminf_{j\to \infty}	\bigg \{\bigg \langle \frac{1}{\e_j}\Big( \int_{t}^{t+\e_j} \frac{1}{2}R(s,t)ds + \int_{t}^{t+\e_j} \frac{1}{2} D^\top \cP_2(s;t) D ds \\&\qq\qq  + \int_{t}^{t+\e_j}  \frac{1}{2} D^\top \cP_8(s;t) Dds  +   \int_{t}^{t+\e_j}  \frac{1}{2}  D^\top \cP_4    N(s,t)	\cP_6     D ds  \Big)  v,v \bigg  \rangle  \\&\qq \q  +  \bigg \langle   \frac{1}{\e_j} \Big\{  \mE_t \int_{t}^{t+\e_j}  \Big[   B^\top \cP_1(s;t) + D^\top \cP_1(s;t)C_\Theta+R(s,t)\Theta\\&\qq\qq  + \big( \widehat{B}^\top +{B}^\top \cP_{10} + D^\top  \cP_{10}  \widehat{D}^\top \big)  \cP_{13}(s;t)   +D^\top  \cP_{10} N(s,t)\cP_{12}C_\Theta \\&\qq\qq  + {B}^\top \cP_{14}(s;t)+ D^\top \cP_{14}(s;t)C_\Theta \Big] X ds\Big\}  , v \bigg  \rangle \bigg\} \geq 0, \q \dbP\mbox{-a.s.}
\end{align*}
This, together with \eqref{s4.2.4-eq4}, \eqref{s4.2-pr1}, \eqref{s4.3-eq1}, \eqref{s4.3-pr2-eq1} and \eqref{s4.3-pr2-eq2}, yields \eqref{s4.1-th1-eq1} and \eqref{s4.1-th1-eq2}.

\ms

{\bf Step 11}. In this step, we prove the ``if" part of Theorem \ref{s4.1-th1}.

\ms

Choose any  sequence $\{\e_j\}_{j=1}^\infty\subset (0,+\infty)$ satisfying  $\lim\limits_{j\to\infty}\e_j=0$. By  \eqref{s4.1-th1-eq1}, \eqref{s4.1-th1-eq2} and \eqref{s4.2-pr1}, for any $t\in [0,T]$, we have 
\begin{equation}\label{s4.3-th1pf-eq1}
\begin{aligned}
\frac{1}{\e_j} \int_{t}^{t+\e_j} \big(	R(s,s) +  D   ^\top \cP_2(s;s) D     +  D    ^\top \cP_8(s;s) D  + D    ^\top \cP_4    N(s,s)	\cP_6     D    \big)ds\geq 0,
\end{aligned}
\end{equation}
and
\begin{equation}\label{s4.3-th1pf-eq2}
\begin{aligned}
&\frac{1}{\e_j} \!\int_{t}^{t+\e_j} \!\!\Big[  B   ^\top \cP_1(s;s)\!+\! D     ^\top \cP_1(s;s)C_\Theta    \! +\! {B}   ^\top \cP_{14}(s;s) \!  \\ &  \qq \q+ D   ^\top \cP_{14}(s;s)C_\Theta     +R(s,s)\Theta     +\big(\widehat{B}    ^\top\!+\! {B}    ^\top \cP_{10}   \!  \\& \qq\q + D     ^\top \cP_{10}     \widehat{D}    ^\top\big) G_2   \cP_{12}     + \ D   ^\top  \cP_{10}    N(s,s)\cP_{12}   C_\Theta    \Big]ds\!=\!0.
\end{aligned}
\end{equation}
On the other hand,	by Assumption  \ref{s4.1-H1} and \eqref{s4.3-pr1-eq1}, it holds that
\begin{align*}
&\bigg| \frac{1}{\e_j} \int_{t}^{t+\e_j} \big( R(s,s) +  D   ^\top \cP_2(s;s) D     +  D   ^\top \cP_8(s;s) D      \\&  +      D   ^\top \cP_4    N(s,s)	\cP_6     D     \big)  ds  - \!\! \frac{1}{\e_j} \!\int_{t}^{t+\e_j} \big( R(s,t) \!+ \! D   ^\top \cP_2(s;t) D    \\&     +  D   ^\top \cP_8(s;t) D      \!  + \!     D   ^\top \cP_4    N(s,t)	\cP_6     D   \big) ds \bigg|_\infty \! \leq \!\mathcal{C}\e_j,
\end{align*}
and
\begin{align*}
&	\bigg| \frac{1}{\e_j} \int_{t}^{t+\e_j} \Big[  B   ^\top \cP_1(s;s)+ D    ^\top  \cP_1(s;s)C_\Theta     + {B}   ^\top \cP_{14}(s;s)   \\ & \qq \qq \q  + D   ^\top \cP_{14}(s;s)C_\Theta     +R(s,s)\Theta    +\big(\widehat{B}   ^\top +  {B}   ^\top \cP_{10}    \\&	\qq \qq \q + D    ^\top \cP_{10}     \widehat{D}   ^\top \big) G_2   \cP_{12}     +   D    ^\top \cP_{10}    N(s,s)\cP_{12}   C_\Theta     \Big]ds \\ &- \frac{1}{\e_j} \int_{t}^{t+\e_j} \Big[  B   ^\top \cP_1(s;t)\Psi(s;t)   + D     ^\top \cP_1(s;t)C_\Theta    \Psi(s;t)  \\ &	\qq\qq\qq  + {B}   ^\top \cP_{14}(s;t) \Psi(s;t)   + D   ^\top \cP_{14}(s;t)C_\Theta   \Psi(s;t)  \\& \qq\qq\qq  + R(s,t)\Theta   \Psi(s;t) +\big(\widehat{B}    ^\top + {B}   ^\top \cP_{10}      + D    ^\top \cP_{10}     \widehat{D}   ^\top\big ) \Gamma(s;t)   \\&\qq\qq\qq +   D    ^\top \cP_{10}    N(s,t)\cP_{12}   C_\Theta   \Psi(s;t) \Big]ds   \bigg|_\infty\\&\leq  \mathcal{C}\e_j+\mathcal{C}\int_{t}^{t+\e_j}\big(1+|\Theta|_\infty\big)ds+\mathcal{C}\frac{1}{\e_j} \int_{t}^{t+\e_j}\big|\Gamma(s;s)-\Gamma(s;t)\big|_\infty ds\\&\leq \mathcal{C}\e_j+\mathcal{C}\int_{t}^{t+\e_j}\big(1+|\Theta|_\infty\big)ds+\frac{\mathcal{C}}{\e_j} \int_{t}^{t+\e_j}\big|\cP_{12}(t)-\cP_{12}  \big|_\infty ds\\&\leq \mathcal{C}\e_j+\mathcal{C}\int_{t}^{t+\e_j}\big(1+|\Theta|_\infty\big)ds+\frac{\mathcal{C}}{\e_j} \int_{t}^{t+\e_j}\!\int_{t}^{s}\big(1+|\Theta(\tau)|_\infty\big )d\tau ds.
\end{align*}
Combining \eqref{s4.3-th1pf-eq1},  \eqref{s4.3-th1pf-eq2} and \eqref{s4.2-Jdelta}, we get
\begin{align*}
&	\liminf_{j\to \infty}	\frac{\cJ(t,x;u^{\e_j}(\cd))-\cJ(t,x;{u(\cd)})}{\e_j}\\&= 	\liminf_{j\to \infty}	\bigg \{\bigg \langle \frac{1}{2\e_j}\Big[ \int_{t}^{t+\e_j} \big(R(s,s) +  D   ^\top \cP_2(s;s) D     +  D   ^\top \cP_8(s;s) D        \\&\q +      D   ^\top \cP_4    N(s,s)	\cP_6     D    \big)ds  \Big]  v,v \bigg  \rangle  +  \bigg \langle  \frac{1}{\e_j} \int_{t}^{t+\e_j} \Big[  B   ^\top \cP_1(s;s)  \\ &\q + D     ^\top \cP_1(s;s)C_\Theta     + {B}   ^\top \cP_{14}(s;s)  + D   ^\top \cP_{14}(s;s)C_\Theta    +R(s,s)\Theta    \\ & \q   +\big(\widehat{B}   ^\top+ {B}   ^\top \cP_{10}    + D     ^\top \cP_{10}     \widehat{D}   ^\top\big) G_2   \cP_{12}    \\&\q+   D   ^\top  \cP_{10}    N(s,s)\cP_{12}   C_\Theta     \Big]ds \,\, x  , v \bigg  \rangle \bigg\} \\&\geq	   \bigg \langle \liminf_{j\to \infty}\frac{1}{2\e_j}\Big[ \int_{t}^{t+\e_j} \big(	R(s,s) +  D   ^\top \cP_2(s;s) D     +  D^\top \cP_8(s;s) D        \\&\q +      D   ^\top \cP_4    N(s,s)	\cP_6     D    \big)ds  \Big]  v,v \bigg  \rangle  +  \bigg \langle  \liminf_{j\to \infty} \frac{1}{\e_j} \int_{t}^{t+\e_j} \Big[  B   ^\top \cP_1(s;s)  \\ &\q + D     ^\top \cP_1(s;s)C_\Theta     + {B}   ^\top \cP_{14}(s;s)   + D   ^\top \cP_{14}(s;s)C_\Theta    + R(s,s)\Theta   \\ & \q  +\big(\widehat{B}   ^\top+ {B}   ^\top \cP_{10}    + D     ^\top \cP_{10}     \widehat{D}   ^\top \big) G_2   \cP_{12}    \\&\q+   D   ^\top  \cP_{10}    N(s,s)\cP_{12}   C_\Theta    \Big]ds \,\, x  , v \bigg  \rangle 
\\&\geq 0, \q \dbP\mbox{-a.s.}
\end{align*}
This completes the proof.
\end{proof}

Now we are in a position to prove Theorem \ref{s4.1-cor1}.

\begin{proof}[Proof of Theorem \ref{s4.1-cor1}] 
	The ``only if" part.   Suppose  there exists an equilibrium strategy $\Theta(\cd)\in L^2(0,T;\dbR^{k\times n})$. By Theorem \ref{s4.1-th1}, we get $({\bf P_1}(\cd;\cd),{\bf P_2}(\cd),{\bf P_3}(\cd;\cd))$ as  the solution of the equation \eqref{s4.1-th1-eq3}.  Let $({ P_1}(\cd;\cd),{ P_2}(\cd),{ P_3}(\cd;\cd))=({\bf P_1}(\cd;\cd),{\bf P_2}(\cd),{\bf P_3}(\cd;\cd))$. From the equation \eqref{s4.1-th1-eq2}, we get \eqref{s4.1-cor1-eq4} and   know that there exists a $\th_0(\cd)\in L^2(0,T;\dbR^{k\times n})$ such that  \eqref{s4.1-cor1-eq2} holds.  From \eqref{s4.1-cor1-eq2}, we find that
	$$
	\begin{aligned}
		&\q \big[R(t,t)+D(t)^\top
		\big({ P_1}(t;t)+ { P_3}(t;t)+{ P_2}(t)^\top N(t,t){ P_2}(t) \big)D(t)\big]^{\dagger}\\ & \q \big[R(t,t)+D(t)^\top
		\big({ P_1}(t;t)+ { P_3}(t;t)+{ P_2}(t)^\top N(t,t){ P_2}(t) \big)D(t)\big]\Theta(t)\\&  =  	\big[R(t,t)+D(t)^\top
		\big({ P_1}(t;t)+ { P_3}(t;t)+{ P_2}(t)^\top N(t,t){ P_2}(t) \big)D(t)\big]^{\dagger}\\&\q \times \big[B(t)^\top\big( { P_1}(t;t)+{ P_3}(t;t)\big)+D(t)^\top
		\big({ P_1}(t;t) +{ P_3}(t;t)+{ P_2}(t) ^\top N(t,t){ P_2}(t) \big)C(t)\\& \qq +\big(\widehat{B}(t)^\top+ {B}(t)^\top { { P_2}}(t)^\top + D (t)^\top { { P_2}}(t) ^\top \widehat{D}^\top(t)\big ) G_2(t){ { P_2}}(t) \big],
	\end{aligned}
	$$
	which implies \eqref{s4.1-cor1-eq3}. Lastly,   \eqref{s4.1-cor1-eq5} follows from \eqref{s4.1-th1-eq1}.
	
	\ss
	
	The ``if" part.   Note that the systems \eqref{s4.1-cor1-eq1} and \eqref{s4.1-th1-eq3} are consistent. If such a $\th_0(\cd)\in L^2(0,T;\dbR^{k\times n})$  exists, then by \eqref{s4.1-cor1-eq3}, we know that the $\Th(\cd)$ defined by \eqref{s4.1-cor1-eq2} is in $ L^2(0,T;\dbR^{k\times n})$. Moreover, by \eqref{s4.1-cor1-eq2},  \eqref{s4.1-cor1-eq4} and \eqref{s4.1-cor1-eq5}, we see that  \eqref{s4.1-th1-eq1} and \eqref{s4.1-th1-eq2} hold. Thus, from Theorem \ref{s4.1-th1}, such $\Theta(\cd)$ is a  closed-loop equilibrium strategy.  
\end{proof}

\section{Proof of Theorem \ref{th2}}\label{sec-proof-main}

In this  section, we prove Theorem \ref{th2}. 
To reach this goal, we first reduce  the  systems \eqref{s4.1-cor1-eq1}--\eqref{s4.1-cor1-eq3}   into an  integral equation system, which is more convenient to be handled.

\begin{proposition}\label{s4.1-cor2}
Let Assumption  \ref{s4.1-H1} hold. Then   Problem (TI-FBSLQ)  admits a closed-loop equilibrium strategy if and only if    there exists  $\theta_0 (\cd) \in L^2(0,T;\dbR^{k\times n})$ such that the following system
\ss
\begin{equation}\label{s4.1-cor2-eq1}
	\begin{cases} \ns \ds 
		\Phi(s,t)=I+\int_{t}^{s} A_{\Th}\Phi(r,t)dr+\int_{t}^{s}C_{\Theta}\Phi(r,t)dW(r),  &\! \!\!\!\!\!\!\!\!\!\!\!\!\!0\leq t \leq s\leq  T,\\\ns \ds  \wt P_2(t)=H+\int_{t}^{T}\big( \wt P_2   A_\Theta   +  \widehat{A}_\Theta     +\widehat{C}   \wt P_2    +\widehat{D}    \wt P_2   C_\Theta    \big)ds,  & 0\!\leq \!t\! \leq\! T,\\ \ns \ds
		\wt{P}_1(t)=\mE_t \Big[ \Phi(T,t)^\top G_1(t)\Phi(T,t)+ \int_{t}^{T} \Phi(s,t)^\top \big( Q(s,t)+\Theta   ^\top R(s,t)\Theta   \\ \ns\ds \qq\qq\q  +  \wt  P_2   ^\top M(s,t)\wt P_2       + C_\Theta    ^\top \wt {P}_2   ^\top N(s,t) \wt {P}_2    C_\Theta      \big)   \Phi(s,t)ds   \Big],  &0\leq t \leq T,
	\end{cases}
\end{equation}
where
\begin{align}\label{s4.1-cor2-eq1.5}
	\nonumber	\Theta(s)\!=&\! -\!\big[R(s,s) \!+ \! D(s)^\top \!
	\big( \wt{P}_1(s)\! + \! \wt P_2(s)^\top N(s,s)\wt P_2(s) \big)D(s)\big]^{\dagger} \\  \ns \ds \nonumber & \times  \big[ B(s)^\top \wt{P}_1(s)     +  D(s)^\top
	\big(\wt{P}_1(s)  + \wt P_2(s)^\top N(s,s)\wt P_2(s) \big)C(s)   \\  \ns \ds \nonumber & + \big(\widehat{B}(s)^\top\! +  {B}(s)^\top   \wt P_{2}(s)^\top   +  D(s)^\top   \wt P_{2}(s)^\top   \widehat{D}(s)^\top\big ) G_2(s)\wt P_{2}(s) \big]   \\\ns\ds  &  +\th_0(s) - \big[ R(s,s)  + D(s)^\top
	\big(\wt{P}_1(s)\! +\! \wt P_2(s)^\top\! N(s,s)\wt P_2(s) \big)D(s)\big]^{\dagger} \\\ns\ds&  \times  \big[R(s,s)   +D(s)^\top
	\big(\wt{P}_1(s)+\wt P_2(s)^\top N(s,s)\wt P_2(s) \big)D(s)\big]\th_0(s),\nonumber
\end{align}
admits a solution $(\wt P_1(\cd),\wt{P}_2(\cd))$ satisfying\ss 
\begin{equation}\label{s4.1-cor2-eq2}
	\!\!\!\!\begin{cases} \ds
		\big[R(\cd,\cd) \! + \! D(\cd)^\top
		\big(\wt{P}_1(\cd) \! + \! \wt P_2(\cd)^\top N(\cd,\cd)\wt P_2(\cd) \big)D(\cd)\big]^{\dagger}  \\\ns\ds \times \big[ \!B(\cd)^\top \wt{P}_1(\cd)    +  D(\cd)^\top
		\big( \wt{P}_1(\cd)   +  \wt P_2(\cd)^\top N(\cd,\!\cd)\wt P_2(\cd) \big)C(\cd)  \\\ns \ds \q +  \big(\widehat{B}(\cd)^\top\!\! +\! {B}(\cd)^\top \wt P_{2}(\cd)^\top \!\!+\! D(\cd)^\top \! \wt P_{2}(\cd)^\top \! \widehat{D}(\cd)^\top\big ) G_2(\cd)\wt P_{2}(\cd) \big] \! \in \! L^2(0,T;\dbR^{k\times n}),\\\ns \ds 
		\mathcal{R}\!\left( R(t,t)\!+\!\! D(t)^\top\!
		\big(\wt{P}_1(t)\!+\!\wt P_2(t)^\top\! N(t,t)\wt P_2(t) \big)D(t) \right) \supseteq \\  \ns \ds     \mathcal{R}\Big( B(t)^\top\! \wt{P}_1(t)  \!   +   D^\top\!(t) 
		\big(\wt{P}_1(t)   +   \wt P_2(t)^\top\! N(t,t)\wt P_2(t) \big)C(t)     \\\ns \ds\q\,  + \big(\widehat{B}(t)^\top  \! +  \! {B}(t)^\top \wt P_{2}(t)^\top  \! +  \! D(t)^\top \wt P_{2}(t)^\top  \widehat{D}(t)^\top\big ) G_2(t)\wt P_{2}(t) \Big), \hspace{4.25em} \ae t\in[0,T], \\\ns \ds 
		R(t,t) +  D(t)^\top \wt{P}_1(t) D(t) + D(t)^\top \wt P_2(t)^\top N(t,t)	\wt P_2(t)  D(t)     \geq 0,\hspace{4em} \ae t\in[0,T].
	\end{cases}
\end{equation}
\end{proposition}	
\begin{proof}
The ``only if" part.  By Theorem \ref{s4.1-cor1}, we know that there exists  $\theta_0 (\cd) \in L^2(0,T;\dbR^{k\times n})$ such that the equations  \eqref{s4.1-cor1-eq1}--\eqref{s4.1-cor1-eq5} admit a solution $(P_1(\cd;\cd),P_2(\cd),P_3(\cd;\cd))$. Let
\begin{equation}\label{s4.1-cor2-pf-eq1}
	\begin{cases}
		\wt{P}(s;t):=P_1(s;t)+P_3(s;t),\\
		\wt{P}_1(s):=\wt P(s;s),\q \wt P_2(s):=P_2(s).
	\end{cases}
\end{equation}
Apply Ito's formula to $s\to   \Phi(s,t)^\top \wt P(s;t)\Phi(s,t)$,  we have
\begin{align*}
&	\wt {P}_1(t)-\Phi(T,t)^\top G_1(t)\Phi(T,t)\\&=\int_{t}^{T}\Phi(s,t)^\top \big(  Q(s,t)+\Theta   ^\top R(s,t)\Theta   +    P_2   ^\top M(s,t) P_2    \\&\q  + C_\Theta    ^\top P_2   ^\top N(s,t)P_2    C_\Theta      \big)\Phi(s,t)ds- \int_{t}^{T} \big(\Phi(s,t)^\top \wt{P}(s,t)C_{\Th}    \Phi(s,t) \\&\q  + \Phi(s,t)^\top   C_{\Th}   ^\top \wt{P}(s,t)\Phi(s,t)   \big)dW(s),
\end{align*}
which implies  the last equation in \eqref{s4.1-cor2-eq1}. Combining  \eqref{s4.1-cor1-eq1}--\eqref{s4.1-cor1-eq5},  we see \eqref{s4.1-cor2-eq2} holds.

\ss

The ``if" part.  Since there exists  $\theta_0 (\cd) \in L^2(0,T;\dbR^{k\times n})$ such that the equations \eqref{s4.1-cor2-eq1}--\eqref{s4.1-cor2-eq2} admit a solution $(\wt P_1(\cd),\wt{P}_2(\cd))$. With $(\wt P_1(\cd),\wt{P}_2(\cd))$ and  $\theta_0 (\cd)$, we choose $\Theta(\cd)$ as in \eqref{s4.1-cor2-eq1.5}, which is in $ L^2(0,T;\mathbb{R}^{k\times n})$ by \eqref{s4.1-cor2-eq2}. Consider the following system:
\begin{equation}\label{11.28-eq5}
	\begin{cases}\ns\ds 
		\frac{dP_1^*(s;t)}{ds} \!+ \! P_{1}^*(s;t) A_\Theta     \!+ A_\Theta   ^\top  P_1^*(s;t) \!  \\ \ns\ds\q + C_\Theta   ^\top  P_{1}^*(s;t)C_\Theta      + Q(s,t) \!+\!\Th^\top R(s,t)\Th\!=\!0,  &   0\!\leq\! t\!\leq\! s\!\leq\! T,\\ \ns\ds 
		\frac{dP_2^*   }{ds}+ P_2^*   A_\Theta   +  \widehat{A}_\Theta     +\widehat{C}   P_2^*    +\widehat{D}    P_2^*   C_\Theta    =0, &  0\!\leq\! t\!\leq\! s\!\leq\! T,\\ \ns\ds
		\frac{dP_3^*(s;t)}{ds}+ P_3^*(s;t)A_\Theta    	+A_\Theta    ^\top P_3^*(s;t)   +C_\Theta    ^\top P_3^*(s;t)C_\Theta    \\ \ns\ds  \q  +{P_2^*}   ^\top M(s,t)P_2^*      + C_\Theta    ^\top {P_2^*}   ^\top N(s,t) 	{P_2^*}    C_\Theta     =0,  &  0\!\leq\! t\!\leq\! s\!\leq\! T,\\\ns\ds 
		P_1^*(T;t)=G_1(t),	P_2^*(T)=H, 	P_3^*(T;t)=0,  & 0\!\leq\! t\!\leq\! T.
	\end{cases}
\end{equation}

Since $P_2^*(\cd)$ and $\wt P_2(\cd)$ solve the same differential equation, we get that $P_2^*(\cd)=\wt P_2(\cd)$ by the uniqueness of the solution to the second equation of \eqref{11.28-eq5}.

Set $\wt P_1^*(s;t):=P_1^*(s;t)+P_3^*(s;t)$.
Similar to the proof of the ``only if" part, we can obtain a representation of $\wt{P}_1^*(t;t)$, and consequently deduce that $ \wt{P}_1^*(t;t)=\wt{P}_1(t)$. Therefore,  $\Theta(\cd)$  can be  rewritten  as
\begin{align*}
	\Theta(s)=& -\big[R(s,s)+D(s)^\top
	\big( P_1^*(s;s)+P_3^*(s;s) +P_2^*(s)^\top N(s,s)P_2^*(s) \big)D(s)\big]^{\dagger}  \\&  \ds\times   \big[B(s)^\top \big(P_1^*(s;s)  +\!P_3^*(s;s)\big) \!  + \! D(s)^\top
	\big(P_1^*(s;s) \!+\!P_3^*(s;s)\!+\! P_2^*(s)^\top N(s,s)P_2^*(s) \big)\\& \q \times C(s)     + \big(\widehat{B}(s)^\top \! + \! {B}(s)^\top P_{2}^*(s)^\top   +\! D(s)^\top  P_{2}^*(s)^\top   \widehat{D}(s)^\top\big ) G_2(s)P_{2}^*(s) \big]  \\& + \th_0(s)-\big[R(s,s)+D(s)^\top
	\big(P_1^*(s;s)+P_3^*(s;s) +\!P_2^*(s)^\top\!\! N(s,\!s)P_2^*(s) \big)D(s)\big]^{\dagger}\\& \times  \big[R(s,\!s)\! + \!\!D(s)^\top\!
	\big(P_1^*(s;\!s)\!+\!P_3^*(s;\!s)\!+\!P_2^*(s)^\top\!\! N(s,\!s)P_2^*(s) \big)D(s)\big]\th_0(s).
\end{align*}
Combining \eqref{s4.1-cor2-eq2} and Theorem \ref{s4.1-cor1}, the sufficiency is proved.
\end{proof}	

Now, we prove Theorem \ref{th2} via Proposition \ref{s4.1-cor2}.

\begin{proof}[Proof of Theorem \ref{th2}]
We divide the proof into four steps.

\ms

\textbf{Step 1}. In this step, we prove the boundedness of $\Theta(\cd)$, $\wt P_1(\cd)$ and $\wt{P}_2(\cd)$ in \eqref{s4.1-cor2-eq1}--\eqref{s4.1-cor2-eq1.5}.

By Assumption \ref{s4.1-H3} and the last equation in \eqref{s4.1-cor2-eq1}, we get that  $\wt P_1(\cd)\geq 0$.

Since $m=n=k=1$,   $\Theta(\cd)$ can be rewritten as
\begin{align*}
	\Th(s)&  = \big(R(s,s) + D(s)^2\wt P_1(s) + D(s)^2N(s,s)\wt P_2(s)^2\big)^{-1} \\ & \q \times\big \{\big(B(s) +  D(s)C(s)\big)\wt{P}_1(s)      +           \widehat{B}(s)G_2(s)\wt P_2(s) \\&\qq + \big[  D(s)C(s)N(s,s) + \big(B(s) + D(s)\widehat{D}(s)\big)G_2(s)\big]      \wt P_2(s)^2\big\} \\&=:I_1(s)+I_2(s),
\end{align*}
where
$$
\begin{aligned}
	|I_1(s)|&=\left|\frac{\big(B(s)+D(s)C(s)\big)\wt{P}_1(s)}{R(s,s)+D (s)^2\wt{P}_1(s)+D(s)^2N(s,s)\wt P_2(s)^2} \right|\\
	& \leq \max\left\{\frac{|B(s)+D(s)C(s)|}{\delta},\frac{|B(s)+D(s)C(s)|}{D(s)^2}\right\}\leq \frac{\cC}{\delta},\q \q \ae s\in [0,T],
\end{aligned}
$$
and
$$
\begin{aligned}
	|I_2(s)|&=\left|\frac{\widehat{B}(s)G_2(s)\wt P_2(s)\! +\!\big[D(s)C(s)N(s,s)\!+\!\big(B(s)\!+\!D(s)\widehat{D}(s)\big)G_2(s)\big]\wt P_2(s)^2}{R(s,s)\!+\! D(s)^2\wt{P}_1(s)\!+\!D(s)^2N(s,s)\wt P_2(s)^2}\right|\\
	& \leq \max\left\{ \frac{\cC}{\delta},\frac{\cC}{D(s)^2\delta} \right\}\leq \max\left\{ \frac{\cC}{\delta},\frac{\cC}{\delta^2} \right\}, \q \ae s\in [0,T].
\end{aligned}
$$
Hence, we have 
$$
|\Theta(t)|<\cC,\qq \ae t\in [0,T].
$$
From the second equation in \eqref{s4.1-cor2-eq1},  we have
\begin{align*}
	|\wt  P_2(t)|&\leq |H|+\int_{t}^{T}\Big(\big|A_{\Theta}     +\widehat{C}     +\widehat{D}     C_{\Theta}      \big| \times\big| \wt P_2     \big|+\big|\widehat{A}_{\Theta}     \big| \Big)ds,
\end{align*}
which, together with Gronwall's inequality, implies that 
$$\wt P_2(s)<\cC,\qq s\in [0,T].$$ 
Similarly, we can get
$$
\ds  \sup_{t\in[0,T]}\sup_{s\in[t,T]}\mE|\Phi(s,t)|^2<\cC.
$$
Then, 
\begin{align*}
	\wt{P}_1(t)\!&=\!\mE  \Big\{ \!\Phi(T,t)^2 G_1(t)\!+\!\! \int_{t}^{T}\!\! \Phi(s,t)^2 \big[ Q(s,t)\!+\!\Theta     ^2 R(s,t) +  (\wt P_2)^2M(s,t)\! +\! C_\Theta      ^2  (\wt P_2)^2N(s,t)   \big]  ds \!   \Big\}\\&=\!G_1(t)\mE\Phi(T,t)^2\!+\!\int_{t}^{T}\! \big[ Q(s,t)\!+\!\Theta     ^2 R(s,t) +  (\wt P_2)^2M(s,t)\! +\! C_\Theta      ^2  (\wt P_2)^2N(s,t)\!\big]  \mE \Phi(s,t)^2 ds\\&<\cC.
\end{align*}

\ms

\textbf{Step 2}. In this step, we establish   estimates for $\wt P_1(\cd)$ and $\wt{P}_2(\cd)$. 

Given $\Theta_{i}$ ($i=1,2$), and
$$
\begin{cases}
	\ds\frac{d\wt P_2^i     }{ds} =-\big(\wt P_2^i     A_{\Theta_i}     +\widehat{A}_{\Theta_i}     +\widehat{C}     \wt P_2^i     +\widehat{D}     \wt P_2^i     C_{\Theta_i}     \big),\q s\in [0,T],\\\ns\ds
	\wt P_2^i(T)=H,
\end{cases}
$$
we have
$$
\begin{aligned}
	&\wt P_2^1(t)-\wt P_2^2(t)\\
	&= \int_{t}^{T}\big[  \big(\wt P_2^1     -\wt P_2^2     \big)A_{\Th_1}      +\wt P_2^2     B      \big(\Theta_1     -\Theta_2     \big)+\widehat{B}     \big(\Theta_1     -\Theta_2     \big)  \\&\qq\q +\widehat{C}      \big(\wt P_2^1     \!-\!\wt P_2^2     \big)  \!+\!\widehat{D}     \big(\wt P_2^1     \!-\!\wt P_2^2     \big)C_{\Th_1}     \!+\!\widehat{D}     \wt P_2^2     D     \big(\Theta_1     \!-\!\Theta_2     \big) \big]ds.
\end{aligned}
$$
By Gronwall's inequality,  we get that
\begin{align}  \label{th1.3-eq1}
\big	|\wt P_2^1(t)-\wt P_2^2(t)\big|\leq& \int_{t}^{T} \big(\cC \big|\wt P_2^1     -\wt P_2^2     \big|+\cC\big|\Theta_1     -\Theta_2     \big|\big)ds\\  \nonumber
	\leq & \cC \int_{t}^{T}\big|\Theta_1     -\Theta_2     \big|ds\\  \nonumber
	\leq & \cC (T-t) \big|\Theta_1(\cd)-\Theta_2(\cd)\big|_{L^\infty(t,T;\dbR^{})},
\end{align}
which yields 
\begin{equation}\label{11.28-eq6}
\sup_{s\in[t,T]}\big|\wt P_2^1(s)-\wt P_2^2(s)\big| \leq \cC (T-t) \big|\Theta_1(\cd)-\Theta_2(\cd)\big|_{L^\infty(t,T;\dbR)}.
\end{equation} 
Similarly,  for  $0\leq t\leq \tau \leq T$, consider the following equation:
$$
\begin{cases}
	d\Phi^i(s,t)=A_{\Th_i}     \Phi^i(s,t)ds+C_{\Th_i}     \Phi^i(s,t)dW(s),\q s\in [t,\tau],\q i=1,2,\\\ns\ds 
	\Phi^i(t,t)=I.
\end{cases}
$$
By standard estimate for SDEs(e.g.,\cite[Theorem 3.2]{Lu-2021}), we have
\begin{align}\label{th1.3-eq2}
	\mE \sup_{s\in[t,\tau]}\big| \Phi^1(s,t)-\Phi^2(s,t)\big|^2&\leq \cC \int_{t}^{\tau}\big|\Theta_1     -\Theta_2     \big|^2ds\\ \nonumber
	&\leq \cC (\tau-t) \big |\Theta_1(\cd)-\Theta_2(\cd)\big|^2_{L^\infty(t,\tau;\dbR^{})},
\end{align}
where the constant $\cC$ is independent of the choice of $t$ and $\tau$. 

For $i=1, 2$, let 
\begin{align*}
	\wt{P}_1^i(t)=&\mE \Big\{\Phi^i(T,t)^\top G_1(t)\Phi^i(T,t)+ \int_{t}^{T} \Phi^i(s,t)^\top \big[ Q(s,t)+\Theta_i     ^\top R(s,t)\Theta_i       \\&\q  + (\wt P_2^i ) ^\top M(s,t)\wt P_2^i      + C_{\Th_i}      ^\top  (\wt P_2^i) ^\top N(s,t)\, \wt P_2^i   \,   C_{\Th_i}        \big]   \Phi^i(s,t)ds   \Big\}.
\end{align*}
From \eqref{th1.3-eq1} and \eqref{th1.3-eq2}, we obtain that
\begin{eqnarray}\label{th1.3-eq3}
	&&	\big|\wt{P}_1^1(t)-\wt{P}_1^2(t)\big | \nonumber\\  &&\leq  \mE \bigg\{ \big|G_1(t) \big( \Phi^1(T,t)^2-\Phi^2(T,t)^2 \big)\big| + \int_{t}^{T} \bigg[ \Big| Q(s,t) \big( \Phi^1(s,t)^2-\Phi^2(s,t)^2 \big) \Big | \nonumber \\   && \q +   \Big| R(s,t)   \big( \Phi^1(s,t)^2\Theta_1     ^2\!-\!\Phi^2(s,t)^2\Theta_2     ^2 \big)   \Big|\! + \! \Big| M(s,t)   \big( \Phi^1(s,t)^2  (\wt P_2^1)^2\!-\!\Phi^2(s,t)^2  (\wt P_2^2)^2 \big)   \Big| \nonumber \\   && \q  + \Big| N(s,t)   \big( \Phi^1(s,t)^2 (\wt P_2^1)^2 C_{\Th_1}     ^2 -\Phi^2(s,t)^2 (\wt P_2^2)^2 C_{\Th_2}     ^2 \big)   \Big| \bigg] ds  \bigg\} \nonumber\\ && \nonumber \leq \cC \bigg\{\Big( \mE\big |  \Phi^1(T,t)-\Phi^2(T,t)\big |^2\Big)^{1/2} + \int_{t}^{T} \bigg[  \Big( \mE\big |  \Phi^1(s,t)-\Phi^2(s,t)\big |^2\Big)^{1/2}  \\  &&  \q +  \Big( \mE\big |  \Phi^1(s,t)\Theta_1     -\Phi^2(s,t)\Theta_2     \big|^2\Big)^{1/2} + \Big( \mE\big |  \Phi^1(s,t)\wt P_2^1     -\Phi^2(s,t)\wt P^2_2     \big|^2\Big)^{1/2} \nonumber\\   && \q +   \Big( \mE\big |  \Phi^1(s,t)\wt P_2^1     C_{\Th_1}     -\Phi^2(s,t)\wt P^2_2     C_{\Th_2}     \big |^2\Big)^{1/2} \bigg]ds   \bigg\}\nonumber
	\\ &&  \leq \cC \bigg[ \Big( \int_{t}^{T} \big|\Theta_1(s)-\Theta_2(s)\big|^2ds \Big)^{1/2}  \\   &&\q   + \int_{t}^{T}  \Big( \int_{t}^{s}  \big|\Theta_1(\tau)-\Theta_2(\tau)\big|^2d\tau+ \big|\Theta_1(s)-\Theta_2(s)\big|^2 + \big |\wt P_2^1(s)-\wt P_2^2(s)\big|^2 \Big)^{1/2} ds  \bigg] \nonumber \\  &&\leq    \,\cC  \Big(\int_{t}^{T} \big|\Theta_1(s)-\Theta_2(s)\big|^2ds\Big)^{1/2} \nonumber \\ &&\leq   \, \cC (T-t)^{1/2} \big|\Theta_1(\cd)-\Theta_2(\cd)\big|_{L^\infty(t,T;\dbR^{}) }.\nonumber
\end{eqnarray}

\ms

\textbf{Step 3}.  Define the mapping $F_0: L^\infty(t_0,T;\dbR^{})\to L^\infty(t_0,T;\dbR^{})$ as follows ($t_0$ is to be determined later): for $\Theta(\cd) \in L^\infty(t_0,T;\dbR^{}) $,
\begin{align*}
	F_0(\Theta(\cd))(s)=	&   -\big[R(s,s) \!+  D(s)^\top 
	\big(\wt{P}_1(s)\! +  \wt{P}_2(s)^\top\! N(s,s)\wt{P}_2(s) \big)D(s)\big]^{\dagger}  \\& \times  \big[ B(s)^\top\! \wt{P}_1(s)   +  D(s)^\top\!
	\big(\wt{P}_1(s)  + \!\wt{P}_2(s)^\top\! N(s,s)\wt{P}_2(s) \big)C(s) \! \\& \q  +  \big(\widehat{B}(s)^\top\!\! + \!  {B}(s)^\top\! \wt{P}_{2}(s)^\top \!\! + \! D(s)^\top \! \wt{P}_{2}(s)^\top  \widehat{D}(s)^\top\big ) G_2(s)\wt{P}_{2}(s) \big]   \\&+  \th_0(s)  -  \big[  R(s,s)  + D(s)^\top\!
	\big(\wt{P}_1(s)    + \wt{P}_2(s)^\top N(s,s)\wt{P}_2(s) \big)D(s)\big]^{\dagger} \\& \times  \big[R(s,s)   + D(s)^\top 
	\big(\wt{P}_1(s) \!+\! \wt{P}_2(s)^\top N(s,s)\wt{P}_2(s) \big)D(s)\big]\th_0(s), \qq \forall \, s\in [t_0,T],
\end{align*} 
where $(\wt{P}_1(\cd), \wt{P}_2(\cd))$ are the solution to the equation \eqref{s4.1-cor2-eq1}  by choosing  $\Theta(\cd)$ to be the control strategy on time interval $[t_1,T]$    in the equation \eqref{s4.1-cor2-eq1.5}.   

Given $\Theta_{i}(\cd)\in L^\infty(t_0,T;\dbR^{})$, $i=1, 2$,  
\begin{eqnarray}\label{pr-th2-eq1}
	&& \!\!\!\!\!\!\!\! F_0(\Theta_1(\cd))(s)-F_0(\Theta_2(\cd))(s)\nonumber\\ \nonumber &&\!\!\!\!\!\!\!\! = \Big\{ \Big[  \big(R(s,s)+D(s)^2\wt{P}_1^1(s)+D(s)^2\wt{P}^1_2(s)^2N(s,s)\big)^{-1}  \\ \nonumber && \q  - \big(R(s,s)+D(s)^2\wt{P}_1^2(s)+D(s)^2\wt{P}^2_2(s)^2N(s,s)\big)^{-1} \Big]  \\ \nonumber &&   \times \Big[  \big(B(s)+D(s)C(s)\big)\wt{P}_1^1(s)+\widehat{B}(s)G_2(s)\wt P^1_2(s) \\ && \q \nonumber+\big[D(s)C(s)N(s,s)  +\big(B(s)+D(s)\widehat{D}(s)\big) G_2(s)\big]\wt{P}^1_2(s)^2  \Big]  \Big\}  \\ &&   +\Big\{  \big(R(s,s)+D(s)^2\wt{P}_1^2(s) + D(s)^2\wt{P}^2_2(s)^2N(s,s)\big)^{-1}  \\ \nonumber &&\q  \times   \Big[ \big(B(s)+D(s)C(s)\big) \big( \wt{P}_1^1(s)- \wt{P}_1^2(s) \big)  +\widehat{B}(s)G_2(s) \big(\wt{P}^1_2(s)- \wt{P}^2_2(s) \big)   \\ \nonumber &&\qq+\big[D(s)C(s)N(s,s)+\big(B(s)+D(s)\widehat{D}(s)\big) G_2(s)\big] \big(\wt{P}^1_2(s)^2-\wt{P}^2_2(s)^2\big) \Big]   \Big \}\\ \nonumber
	&&\!\!\!\!\!\!\!\! \leq \frac{\cC}{\delta^2} \big( \big|\wt{P}_1^1(s)\!-\!\wt P_1^2(s)\big|  \! + \! \big| P^2_2(s)^2\!-\!P^1_2(s)^2 \big|  \big) \! \\&& \nonumber  + \! \frac{\cC}{\delta} \big(\big| \wt{P}_1^1(s)\!-\! \wt P_1^2(s) \big|\! +\! \big|\wt{P}^2_2(s)\!-\!\wt{P}^1_2(s) \big|\!  +\!\big | \wt{P}^2_2(s)^2\!-\!\wt{P}^1_2(s)^2 \big|  \big)\\  \nonumber
	&&\!\!\!\!\!\!\!\! \leq \cC \big( \big| \wt{P}_1^1(s)- \wt{P}_1^2(s) \big| + \big|\wt{P}^1_2(s)-\wt{P}^2_2(s)\big |  +\big | \wt{P}^1_2(s)^2-\wt{P}^2_2(s)^2 \big|   \big).
\end{eqnarray}
By \eqref{11.28-eq6} and \eqref{th1.3-eq3}, we get that
\begin{equation}\label{11.28-eq7}
\big| F_0(\Theta_1(\cd))(\cd)-F_0(\Theta_2(\cd))(\cd)\big|_{L^\infty(t_0,T;\dbR^{})} \leq  \cC(T-t_0)^{1/2} \big| \Theta_1(\cd)-\Theta_2(\cd)  \big|_{L^\infty(t_0,T;\dbR^{})}. 
\end{equation}
Choose $t_0$ such that $\cC (T-t_0)^{1/2}<1/2$, via Banach fixed-point theorem, we know there is a unique fixed point $\Theta^*(\cd)\in L^\infty(t_0,T;\dbR^{})$ for the mapping $F_0$. 

\ms

\textbf{Step 4}.  Define the mapping $F_1: L^\infty(t_1,t_0;\dbR^{})\to L^\infty(t_1,t_0;\dbR^{})$ as follows  ($t_1$ is to be determined later): for $\Theta(\cd) \in L^\infty(t_1,t_0;\dbR^{}) $,
\begin{align*}
		F_1(\Theta(\cd))(s) =&  \Big\{ \! -\! \big[R(s,s) \!+ \! D(s)^\top \!
	\big( \wt{P}_1(s)\!\! + \! \wt{P}_2(s)^\top N(s,s)\wt{P}_2(s) \big)D(s)\big]^{\dagger}  \\& \times \big[ B(s)^\top \wt{P}_1(s)  \!\! + \! D(s)^\top
	\big(\wt{P}_1(s)  +\! \wt{P}_2(s)^\top\! N(s,s)\wt{P}_2(s) \big)C(s)\!\\&  \q +   \big(\widehat{B}(s)^\top\!\! +\! B(s)^\top\!\! \wt{P}_{2}(s)^\top\!\! +   D(s)^\top\! \wt{P}_{2}(s)^\top\! \widehat{D}(s)^\top\big ) G_2(s)\wt{P}_{2}(s) \big]\\&   +  \th_0(s)  -  \big[ R(s,s)  + D(s)^\top
	\big(\wt{P}_1(s) + \wt{P}_2(s)^\top N(s,s)\wt{P}_2(s) \big)D(s)\big]^{\dagger}\\& \times    \big[R(s,s) \! + \! D(s)^\top \!
	\big(\wt{P}_1(s)\!+\!\wt{P}_2(s)^\top \!N(s,s)\wt{P}_2(s) \big)D(s)\big]\th_0(s) \Big \}\! \times\! \chi_{[t_1,t_0]}(s) ,\, s\in [t_1,t_0],
\end{align*} 
where $(\wt{P}_1(\cd), \wt{P}_2(\cd))$ is the solution of the equation \eqref{s4.1-cor2-eq1} by choosing  $\Theta(\cd)\chi_{[t_1,t_0]}(\cd)+\Theta^*(\cd)\chi_{[t_0,T]}(\cd)$ to be the control strategy on time interval $[t_1,T]$  in the equation \eqref{s4.1-cor2-eq1.5}. 

Given $\Theta_{i}(\cd)\in L^\infty(t_1,t_0;\dbR^{})$, $i=1,2$, similar to the proof of  \eqref{11.28-eq7}, we can show that
$$
\big| F_1(\Theta_1(\cd))(\cd)-F_1(\Theta_2(\cd))(\cd)\big|_{L^\infty(t_1,t_0;\dbR^{})} \leq  \cC(t_0-t_1)^{1/2} \big| \Theta_1(\cd)-\Theta_2(\cd)  \big|_{L^\infty(t_1,t_0;\dbR^{})}. 
$$ 
By choosing $t_0-t_1= T-t_0$,  it follows from Banach fixed-point theorem that there is a unique fixed point $\Theta^*(\cd)\in L^\infty(t_1,t_0;\dbR^{})$ for the mapping $F_1$.  Inductively, we   obtain the existence and uniqueness of the solution to the equations \eqref{s4.1-cor2-eq1}--\eqref{s4.1-cor2-eq2}. By  Proposition \ref{s4.1-cor2} and Theorem \ref{s4.1-cor1}, we complete the proof.
\end{proof}

\end{document}